\documentclass[11pt,twoside,english]{article}
\usepackage{amsmath}
\usepackage{amsfonts}
\usepackage{mathrsfs}
\usepackage{color}
\usepackage{ amsmath, amsfonts, amssymb, amsthm, amscd}
\usepackage[T1]{fontenc}
\usepackage[english]{babel}
\usepackage[noinfoline]{imsart}
\usepackage{graphicx}
\usepackage{stmaryrd}  %\interleave
\usepackage{titletoc}  %\ttl

\newtheorem{theorem}{Theorem}
\newtheorem{lemma}{Lemma}
\newtheorem{proposition}{Proposition}

\newtheorem{definition}{Definition}
\newtheorem{corollary}{Corollary}
\newtheorem{remark}{Remark}
\newtheorem{assumption}{Assumption}
\newcommand{\be}{\begin{equation}}
\newcommand{\ee}{\end{equation}}
\newcommand{\beq}{\begin{eqnarray}}
\newcommand{\eeq}{\end{eqnarray}}

\newcommand{\eps}{\varepsilon}

\newcommand{\ced}{\end{proof}}

 \allowdisplaybreaks[3]

\begin{document}
\begin{frontmatter}
\title{A Probabilistic Interpretation of the Master Equation Arising from Mean Field Games with Jump Diffusion}
\date{}
\runtitle{}
\author{\fnms{Jisheng}
 \snm{LIU}\corref{}\ead[label=e1]{jsliu23@m.fudan.edu.cn}}
\address{School of Mathematical Sciences\\
Fudan University, China \\
\printead{e1}}

\author{\fnms{Jing}
 \snm{ZHANG}\corref{}\ead[label=e2]{zhang\_jing@fudan.edu.cn}}
\thankstext{T2}{The work is supported by National Key R\&D Program of China (2022YFA1006101), National Natural Science Foundation of China (12271103, 12031009, 12326603) and Shanghai Science and Technology Commission Grant (21ZR140860).}
\address{School of Mathematical Sciences \\
Fudan University, China  \\
\printead{e2}}

\runauthor{J. Liu and J. Zhang}

\begin{abstract}
In this paper we study the classical solution to the master equation arising from mean-field games (MFGs) driven by jump-diffusion processes. The master equation, a nonlinear partial differential equation on Wasserstein space, characterizes the value function of MFGs and is challenging to analyze directly due to its measure-valued derivatives. We propose a probabilistic interpretation using coupled McKean–Vlasov forward-backward stochastic differential equations (MV-FBSDEs) with jumps. Under suitable Lipschitz and differentiability assumptions on the coefficients, we first establish the well-posedness of the MV-FBSDEs on a small time interval via a contraction mapping argument. We then prove the existence and regularity of the first- and second-order derivatives of the solutions with respect to the spatial and measure variables, relying on careful estimates involving jump terms and measure derivatives. Finally, we show that the decoupling field of the MV-FBSDEs satisfies the master equation in the classical sense, providing both existence and uniqueness of the solution. Our work extends earlier results on diffusion-driven MFGs to the jump-diffusion setting and offers a probabilistic framework for analyzing and numerically solving such kind of master equations.
\end{abstract}


\end{frontmatter}

%
%
%
%
%\title{The obstacle Problem for Quasilinear Stochastic PDEs: Analytical  approach}
%
%\author{\first \and \and \  \second \and\third}
%
%\maketitle
%
%
%
%\begin{abstract}
%We prove  existence and uniqueness of the
%solution  of quasilinear stochastic PDEs with obstacle. Our method is based on analytical technics coming from the parabolic potential theory. The solution
%is expressed as  pair $(u,\nu)$  where $u$ is a predictable
%continuous process which takes values in a proper Sobolev space and
%$\nu$ is a random regular measure satisfying minimal Skohorod
%condition.
%\end{abstract}
%
%
%
%\paragraph{Key words and phrases :}
%parabolic potential, regular measure, stochastic partial differential equations, obstacle
%problem, penalization method, Ito's formula, comparison theorem.
%
%\paragraph{MSC :} 60H15; 35R60; 31B15
\section{Introduction}

%1.主方程整体介绍；2.平均场问题介绍，包括平均场控制问题和平均场博弈问题；3.跳过程介绍，主要是现有的带跳的伊藤公式以及控制问题；4.feynman-kac公式介绍，介绍经典解的结论，以及我们的结论的好处：联系、计算；5.文章内容介绍

%本文主要提供了一类带平均场的主方程的概率表示方法，也即Feynamn-Kac公式。这类主方程主要来自于平均场博弈问题以及平均场控制问题的研究。
This paper is devoted to providing a probabilistic interpretation of the following mean-field master equations, which come from the study of mean-field game problems driven by jump-diffusion processes,
%This paper is concerned with providing a probabilistic interpretation of the following mean-field master equations, which come from the study of mean-field game problems driven by jump-diffusion processes,
\begin{equation}\label{eq master introduction game general}\begin{split}
\partial_t u(t,x,\mu)=Au(t,x,\mu)+f(t,x,\mu,u(t,x,\mu),&Bu(t,x,\mu))+\int_{\mathbb{R}^d}\, [Cu(t,x,\mu)](\theta)d\mu(\theta),\\& (t,x,\mu)\in[0,T]\times\mathbb{R}^d\times\mathcal{P}_2(\mathbb{R}^d),
\end{split}\end{equation}
where 
%$\mathcal{P}_2(\mathbb{R}^d)$ is the Wasserstain space of probability measure on $\mathbb{R}^d$ with finite second-order moment. 
$A$ and $B$ are the second- and first-order differential operators w.r.t. $x$, respectively, while $C$ is a non-local operator involving differentiation w.r.t. $\mu$.
%近年来，平均场问题是一类热门问题，它主要是由多粒子系统逼近而来。平均场问题的主要特点是，其相关系数有一个测度变量，也即状态过程的变化与状态的分布有关。一方面，平均场博弈问题由Caines和Lasry分别开始研究，并由Lions，cardaliaguet等进一步发展，他们也发展了主方程的相关理论，使其成为了研究平均场博弈问题的有力工具。进一步，在最近的研究中，mou等通过Fockker-Planck方程建立主方程对应的FBSDE，证明了随机项较简单的情形下，各种系数条件的主方程适定性问题。另一方面，平均场控制问题，也称McKean-Vlasov控制问题，考虑的是针对McKean-Vlasov过程的控制问题。其研究方法仍然是采用传统的两种解决控制问题的办法：【参考xiaoli的介绍】。其中，利用动态规划原理所得到的HJB方程也可视作一类主方程，Qiu等给出了这个方程的粘性解的适定性。
\\In recent years, mean-field problems have emerged as an extensively studied class of problems, primarily arising from multi-particle systems. A key characteristic is that their dynamics depend on a measure-valued variable, meaning that the evolution of the state processes is inherently coupled with their distribution.
Mean-field game (MFG) theory was initiated independently by Caines et al. \cite{huang2006large} and Lasry and Lions \cite{lasry2007mean}. Subsequently, the theory was advanced by Lions \cite{lions2007cours}, Cardaliaguet et al. \cite{cardaliaguet2010notes}, and Carmona and Delarue \cite{carmona2014master,carmona2018probabilistic}. A key outcome of this development is the master equation, which was established as a powerful tool for analyzing MFG. It takes the following form:
\begin{align*}
0&=\partial_t u(t,x,\mu)+\frac{1}{2}\text{Tr}\Big[(\sigma\sigma^\intercal )(t,x)\partial_x^2u(t,x,\mu)\Big]+H(t,x,\mu,D_xu(t,x,\mu))\\
&+\int_{\mathbb{R}^d}\bigg[\frac{1}{2}\text{Tr}\Big[(\sigma\sigma^\intercal )(t,x)\partial_v\partial_\mu u(t,x,\mu,v)\Big]+\partial_\mu u(t,x,\mu,v)\cdot \partial_p H(t,v,\mu,D_x u(t,v,\mu))\bigg]d\mu(v),
\end{align*}
with terminal condition $u(T,x,\mu)=G(x,\mu)$, where $H$ denotes the Hamiltonian function.
%where $u=u(t,x,\mu):[0,T]\times \mathbb{R}^d\times\mathcal{P}_2(\mathbb{R}^d)\to\mathbb{R}^d$.\\
\\Furthermore, a connection between master equations and forward-backward stochastic differential equations (FBSDEs) via the Fokker–Planck equation has been established in recent work by Cardaliaguet et al. \cite{cardaliaguet2019master} and Mou et al. \cite{meszaros2024mean,mou2024mean}. In these studies, the well-posedness of second-order master equations with only a diffusion term  was demonstrated under various assumptions on the coefficients.
\\To address the challenges concerning classical solutions and numerical computation for master equations, Chassagneux et al. \cite{chassagneux2014probabilistic} provided a probabilistic representation of their classical solutions. This approach was motivated by the analytical difficulties inherent to master equations, which stem from their measure-dependent derivatives. In contrast, McKean-Vlasov forward-backward stochastic differential equations (MV-FBSDEs) are more tractable and have some existing results.
Under sufficiently smooth coefficients, they proved the well-posedness (existence, uniqueness, and stability) of solutions to the relevant MV-FBSDEs, including their measure derivatives. The connection between these probabilistic solutions and the classical solutions of the master equation was then rigorously established by using the It\^o-Lions formula. Finally, based on this theoretical foundation, they constructed numerical algorithms for computing the classical solutions of master equations, see \cite{chassagneux2019numerical}. 
To the best of our knowledge, however, this probabilistic and numerical framework has not been applied to mean-field games driven by jump-diffusion processes.\\
% However, an explicit Feynman-Kac formula for viscosity solutions of master equations has not yet been proposed to the best of our knowledge. {\red Since classical solutions may not exist when there is a jump diffusion which cause the nonlinear term of the derivative of the measure variable (see in Section 2.4)}, the probabilistic representation of viscosity solutions holds significant value. Not only did it provide theoretical insights, but also paved the way for developing numerical methods tailored to viscosity solutions of master equations.\\
Recently, Guo et al. \cite{guo2023ito} established a jump-diffusion It\^o-Lions formula and derived the corresponding Hamilton-Jacobi-Bellman equation for mean-field control problems driven by jump-diffusion processes as follows,
\begin{equation*}\left\{
\begin{split}
&\partial_t u(t,\mu)+\mathbb{E}\Big[\inf_{a\in A}H(t,\xi,a,\mu,\partial_\mu u(t,\mu,\xi),\partial_v\partial_\mu u(t,\mu,\xi),\frac{\delta u}{\delta\mu}(t,\mu,\xi))\Big]=0,\\
&u(T,\mu)=G(\mu),
\end{split}\right.\end{equation*}
where $\mu=\mathbb{P}_\xi$ and
\begin{align*}
H(t,x,a,\mu,p,M,r(\cdot)):=&f(t,x,\mu,a)+b(t,x,\mu,a)\cdot p+\frac{1}{2}M:\sigma^{\intercal}\sigma(t,x,\mu,a)\\
&+\int_E\, (r(x+h(t,x,\mu,a,\theta))-r(x))d\nu(\theta).
\end{align*}
Compared to the master equation discussed above, this HJB equation introduces an additional term that arises from the structure of the jump-diffusion It\^o-Lions formula (see \cite{guo2023ito}). Consequently, it is foreseeable that if the state equation of mean-field games is driven by jump-diffusion processes, the associated master equation will adopt a similar form (see Section 2.4). 
Moreover, the jump-diffusion It\^o-Wentzell-Lions formula developed in \cite{jisheng2025wentzell} provides the necessary tool to study mean-field control/game problems involving both random coefficients and jumps, enabling the derivation of corresponding stochastic master equations, which we will study in our future work.\\
%本文我们会研究更一般的结论，我们将给出带跳过程导出的主方程的粘性解的feynman-kac公式。首先我们通过FBSDE的经典方法给出在lipschitz系数条件下带跳MV-FBSDE的适定性。相比经典的证明方法，我们需要额外处理测度相关的估计以及跳过程项的估计。作为结果，我们给出了小时间条件以及Lipschitz系数条件来保证带跳MV-FBSDE的适定性。然后，用自然的方法定义了主方程的粘性解后，利用tang的框架我们给出了feyman-kac公式。
In this paper, we investigate the following master equations for jump-diffusion processes via a probabilistic approach and establish the well-posedness of their classical solutions:
\begin{align*}
&\partial_tV(t,x,\mu)+\partial_x V(t,x,\mu)b(t,x,\mu,V(t,x,\mu),\partial_xV(t,x,\mu)\sigma(t,x,\mu,V(t,x,\mu)))\nonumber\\
&+\frac{1}{2}\partial_x^2V(t,x,\mu):\sigma\sigma^{\intercal}(t,x,\mu,V(t,x,\mu))\\
&+\int_E\, (V(t,x+h(t,x,\mu,V(t,x,\mu),\theta),\mu)-V(t,x,\mu))\nu(d\theta)\nonumber\\
&+\int_{\mathbb{R}^d}\, \bigg[\partial_\mu V(t,x,\mu,y)b(t,y,\mu,V(t,y,\mu),\partial_xV(t,y,\mu)\sigma(t,y,\mu,V(t,y,\mu)))\nonumber\\
&\qquad\qquad+\frac{1}{2}\partial_y\partial_\mu V(t,x,\mu,y):\sigma\sigma^{\intercal}(t,y,\mu,V(t,y,\mu))\nonumber\\
&\qquad\qquad+\int_E\Big(\frac{\delta V}{\delta\mu}(t,x,\mu,y+h(t,y,\mu,V(t,y,\mu),\theta))-\frac{\delta V}{\delta\mu}(t,x,\mu,y)\Big)\nu(d\theta)\bigg]\mu(dy)=0,
\end{align*}
with terminal condition $V(T,x,\mu)=g(x,\mu)$ for any $\mu\in\mathcal{P}_2(\mathbb{R}^d),x\in\mathbb{R}^d$.\\
The well-posedness of McKean-Vlasov stochastic differential equations (SDEs) driven by Poisson random measures, along with the existence of solution derivatives, was first established by Hao and Li \cite{hao2016mean}. Subsequently, Li \cite{li2018mean} extended these results to the backward case, proving the existence and uniqueness of solutions and their derivatives for McKean-Vlasov backward differential equations (BSDEs) with Poisson jumps. For the coupled forward-backward system with jumps, however, analogous results are absent. This gap forms the essential foundation needed to support a probabilistic interpretation of the master equation, which we address in this work.
First, we establish the well-posedness of MV-FBSDEs driven by jump-diffusion processes. Compared to conventional approaches, our analysis requires systematically handling measure-dependent estimates and the terms arising from the jump process. To overcome these additional complexities, we derive small-time horizon conditions that guarantee well-posedness. Next, to provide a probabilistic interpretation for the classical solutions, we need to prove the existence of derivatives of the MV-FBSDEs with respect to their initial values. The key challenges here involve obtaining the $2p$-estimates for solutions of jump-diffusion MV-FBSDEs and formulating the equations governing these derivatives. For the latter, our approach adapts the methodology from \cite{li2018mean}—originally developed for decoupled MV-FBSDEs—to our coupled setting.\\
It is important to note that we impose a different regularity condition on the jump coefficient from that in \cite{hao2016mean}. In their work, to accommodate the potential infinity of the intensity measure $\nu$, the coefficient $h$ is required to satisfy the following  boundedness and Lipschitz conditions:
\begin{align*}
|h(t,x,\mu,a,\theta)|\leq& L(1\wedge|\theta|),\\
|h(t,x_1,\mu_1,a,\theta)-h(t,x_2,\mu_2,a,\theta)|\leq &L(1\wedge|\theta|)(|x_1-x_2|+W_2(\mu_1,\mu_2)).
\end{align*}
In contrast, our framework assumes the intensity measure $\nu$ to be finite. This enables us to adopt a more general condition, similar to the one used in \cite{chen2015semi}.\\
The remainder of this paper is organized as follows. Section 2 introduces the necessary notations, the concept of measure derivatives, Poisson random measures, and formulates the master equation for MFGs with jump-diffusion. In Section 3, we augment the standard Lipschitz assumptions with a small-time horizon condition. Within this framework, we establish the well-posedness (existence and uniqueness) of solutions to the associated MV-FBSDEs. Sections 4 and 5 are devoted to proving the existence of first- and second-order derivatives of the solutions to the (MV-)FBSDEs, respectively. Finally, Section 6 establishes the equivalence between the classical solution of the master equation and the solution of the coupled (MV-)FBSDE system.

\section{Preliminary}
\subsection{Notations}
Let $(\Omega,\mathscr{F},\mathbb{F}:=(\mathscr{F}_t)_{t\geq0},\mathbb{P})$ be a standard filtered probability space satisfying the usual conditions. For a random variable $X$ defined on $(\Omega,\mathscr{F},\mathbb{P})$, we denote its distribution as  $\mathbb{P}_X$.
Let $\mathcal{P}(\mathbb{R}^d)$ denote the set of all probability measures on $\mathbb{R}^d$, and the $p$-Wasserstein space of measures on $\mathbb{R}^d$ is defined as $$\mathcal{P}_p(\mathbb{R}^d):=\left\{\mu\in\mathcal{P}(\mathbb{R}^d)\bigg|\left\lVert \mu\right\rVert_p:=\bigg(\int_{\mathbb{R}^d}|x|^p\mu(dx)\bigg)^{\frac{1}{p}} <\infty\right\},$$ 
equipped with the distance $$W_p(\mu,\nu):=\inf_{\pi\in\Pi_{\mu,\nu}}\bigg(\int_{\mathbb{R}^d\times\mathbb{R}^d}  \,\left\lvert x_1-x_2\right\rvert^p \pi(dx_1,dx_2)\bigg)^{\frac{1}{p}},\quad \forall\mu,\nu\in \mathcal{P}_p(\mathbb{R}^d),$$ 
where $\Pi_{\mu,\nu}$ is the set of all the transport from $\mu$ to $\nu$. %We can also see $\Pi_{\mu,\nu}$ as a set of measures in $\mathcal{P}(\mathbb{R}^d\times\mathbb{R}^d)$ (the set of all joint probability measure on $\mathbb{R}^d\times\mathbb{R}^d$) with marginals $\mu$ and $\nu$.
\\For vectors $a,b$ in $\mathbb{R}^d$, $ab=a\cdot b:=\sum_{i=1}^d a_i b_i$, where $a_i,b_i$ are the $i$th components of $a,b$ respectively. For matrix $A,A'\in \mathbb{R}^{d\times d}$, the trace of $A$ is defined by $Tr(A):=\sum_{i=1}^d A_{i,i}$. Here, for convenience, we write $Tr(A\cdot A')$ as $A:A'$. We denote by $|a|$ the $\mathbb{R}^d$ norm of the vector $a$ and by $\|A\|$ the Frobenius norm of the matrix $A$.\\
The space $L^2(\Omega;\mathbb{R}^d)$ is the set of all $\mathbb{R}^d-$valued $\mathscr{F}-$measurable and square integrable random variables. We denote by $\mathcal{S}_{\mathbb{F}}^2(0,T;\mathbb{R}^d)$ the set of all $\mathbb{R}^d-$valued $\mathbb{F}-$predictable c{\`a}dl{\`a}g such that $$\|X\|^2_{\mathcal{S}^2_\mathbb{F}(0,T;\mathbb{R}^d)}:=\mathbb{E}\sup_{t\in[0,T]}|X_t|^2 <\infty$$ 
and by $\mathcal{M}^2_\mathbb{F}(0,T;\mathbb{R}^d)$ the class of $\mathbb{R}^d-$valued $\mathbb{F}-$predictable c{\`a}dl{\`a}g such that 
$$\|X\|^2_{\mathcal{M}^2_{\mathbb{F}}(0,T;\mathbb{R}^d)}:=\mathbb{E}\int_0^T|X_t|^2dt<\infty.$$
Moreover, by $\mathcal{K}^2_\nu(0,T;\mathbb{R}^d)$ we denote the class of $\mathbb{R}^d-$valued $\mathcal{P}_0\times \mathcal{B}(E)$-measurable functions with parameter $\theta\in E$ such that
$$\|X\|_{\mathcal{K}^2_\nu(0,T;\mathbb{R}^d)}^2:=\mathbb{E}\int_0^T \int_E |X_t(\theta)|^2\nu(d\theta)dt<\infty,$$
for some measure $\nu$. Here $\mathcal{P}_0$ denotes the $\sigma$-field of $\mathbb{F}-$predictable subsets of $\Omega\times[0,T]$. 
When there is no confusion, we omit $\mathbb{R}^d$ from the notation.
\\For any $k=0,1,\cdots$, $C^k(E;\mathbb{R}^d)$ is the set of all $k-$th differentiable functions $f:E\to\mathbb{R}^d$ with continuous derivatives and
$C^k_b(E;\mathbb{R}^d)$ denotes the set of functions $f\in C^k(E;\mathbb{R}^d)$ such that its derivatives and itself are bounded. When $k=0$, we simply denote them as $C(E;\mathbb{R}^d)$ and $C_b(E;\mathbb{R}^d)$, respectively.\\
By saying that $\tilde{X}$ is an independent copy of the random variable $X$ defined on $(\Omega,\mathscr{F},\mathbb{P})$, we mean that $\tilde{X}$ is defined on another probability space $(\tilde{\Omega},\tilde{\mathscr{F}},\tilde{\mathbb{P}})$, has the same law as $X$, and is independent of $X$ when both are considered on the product space $(\Omega\times\tilde{\Omega}, \mathscr{F}\otimes\tilde{\mathscr{F}}, \mathbb{P}\otimes\tilde{\mathbb{P}})$.
More precisely, on the product space we define $X(\omega,\tilde{\omega}):=X(\omega)$ and $\tilde{X}(\omega,\tilde{\omega}):=X'(\tilde{\omega})$, where $X'$ is a random variable on $(\tilde{\Omega},\tilde{\mathscr{F}},\tilde{\mathbb{P}})$ with $\mathcal{L}(X')=\mathcal{L}(X)$.
Here $\tilde{\mathbb{E}}[\cdot]$ denotes the expectation with respect to $\tilde{\mathbb{P}}$ on the auxiliary space $\tilde{\Omega}$. Similarly, $\bar{X}$ denotes another independent copy of $X$, defined on a further space $(\bar{\Omega},\bar{\mathscr{F}},\bar{\mathbb{P}})$.

\subsection{Differentiability of functions w.r.t. probability measures}
In this paper, we will work with two types of measure derivatives: the Lions derivative, defined via the Fr\'echet derivative in a Hilbert space, and the linear derivative, defined intrinsically on $\mathcal{P}_2(\mathbb{R}^d)$. This subsection recalls the necessary definitions and relevant propositions. For the details, we refer the readers to, for example, \cite{cardaliaguet2010notes,cardaliaguet2019master,carmona2018probabilistic,gangbo2019differentiability}.\\
Assume that $\mathscr{F}$ is atomless (i.e., $\forall E\in \mathscr{F}$ with $\mathbb{P}(E)>0$, there exists $E'\subset E$ in $\mathscr{F}$ s.t. $0<\mathbb{P}(E')<\mathbb{P}(E)$). For any map $U:\mathcal{P}_2(\mathbb{R}^d)\to\mathbb{R}$, we define its lift $\widetilde{U}$ on $L^2(\Omega;\mathbb{R}^d)$ by
\begin{eqnarray*}
\widetilde{U}(X):=U(\mathbb{P}_{X}) ,\ \ \ \forall X\in L^2(\Omega;\mathbb{R}^d).
\end{eqnarray*}
Since $X\in L^2(\Omega;\mathbb{R}^d)$, we have $\mathbb{P}_{X}\in \mathcal{P}_2(\mathbb{R}^d)$ by definition, which makes the lift meaningful.\\
The Lions derivative is defined by exploiting the fact that $L^2(\Omega;\mathbb{R}^d)$ is a Hilbert space, where Fr{\'e}chet differentiability is well-defined. Precisely, 
$\widetilde{U}$ is said to be Fr{\'e}chet differentiable at $X_0$ if there exists a linear continuous map $D\widetilde{U}(X_0):L^2(\Omega;\mathbb{R}^d)\to\mathbb{R}$ such that
\begin{eqnarray*}
\widetilde{U}(X)- \widetilde{U}(X_0)=\mathbb{E}[D\widetilde{U}(X_0)(X-X_0)]+o(\left\lVert X-X_0\right\rVert_{L^2} ), \ \mbox{as}\ \left\lVert X-X_0\right\rVert_{L^2}\to 0.
\end{eqnarray*}
It is shown in \cite{carmona2018probabilistic} (Proposition 5.25) that the law of $D\widetilde{U}(X_0)$ depends on $X_0$ only via its law $\mathbb{P}_{X_0}$, and there exists a Borel function $h_{\mathbb{P}_{X_0}}:\mathbb{R}^d\to\mathbb{R}^d$ satisfying
\begin{eqnarray}\label{lionsderivative}
D\widetilde{U}(X_0)=h_{\mathbb{P}_{X_0}}(X_0).
\end{eqnarray}
%Thus the definition of Lions derivative is as following.
\begin{definition}\label{Lionsderiv}
$U$ is differentiable at $\mu_0:=\mathbb{P}_{X_0}\in \mathcal{P}_2(\mathbb{R}^d)$ if its lift $\widetilde{U} $ is Fr{\'e}chet differentiable at $X_0$. In this case, 
we define the function $h_{\mathbb{P}_{X_0}}$ shown in (\ref{lionsderivative}) as the Lions derivative of $U$ at $\mu_0$, denoted by $\partial_{\mu}U(\mu_0,\cdot)$.
\end{definition}
We say that the function $U:\mathcal{P}_2(\mathbb{R}^d)\to\mathbb{R}$ is of class $\mathscr{C}^1(\mathcal{P}_2(\mathbb{R}^d))$ if it is Lions differentiable and its Lions derivative $\partial_\mu U(\mu,y)$ is jointly continuous in $(\mu,y)$. Furthermore, we say that $U$ is in $\mathscr{C}^{1,1}(\mathcal{P}_2(\mathbb{R}^d))$ if $U\in \mathscr{C}^1(\mathcal{P}_2(\mathbb{R}^d))$ and for any $\mu\in\mathcal{P}_2(\mathbb{R}^d)$, the map $y\mapsto \partial_\mu U(\mu,y)$ is differentiable, and the derivative $\partial_y\partial_\mu U(\mu,y)$ is jointly continuous in $(\mu,y)$.
\\Now we introduce the linear derivative on Wasserstain space $\mathcal{P}_2(\mathbb{R}^d)$.
\begin{definition}\label{linearderiv}
We say that $U:\mathcal{P}_2(\mathbb{R}^d)\to \mathbb{R}$ is linear differentiable if there exists a jointly continuous and bounded map $\frac{\delta U}{\delta\mu}: \mathcal{P}_2(\mathbb{R}^d)\times\mathbb{R}^d\to \mathbb{R}$ such that
\begin{eqnarray*}
U(\mu')-U(\mu)=\int_{0}^{1}  \,\int_{\mathbb{R}^d} \,\frac{\delta U}{\delta\mu}((1-h)\mu+h\mu',y)(\mu'-\mu)dy dh, \quad \forall \mu,\mu'\in\mathcal{P}_2(\mathbb{R}^d).
\end{eqnarray*}
Then $\frac{\delta U}{\delta\mu}$ is called a linear derivative. Moreover, we adopt the normalization convention
\begin{eqnarray*}
\int_{\mathbb{R}^d} \,\frac{\delta U}{\delta\mu}(\mu,y)\mu(dy)=0,\quad\forall\mu\in\mathcal{P}_2(\mathbb{R}^d).
\end{eqnarray*}
%If $U$ and its linear derivative are both uniformly bounded (w.r.t. $y$), we say that $U$ is of class $\mathscr{C}^1_b(\mathcal{P}_2(\mathbb{R}^d))$. Moreover, if $U\in\mathscr{C}^1(\mathcal{P}_2(\mathbb{R}^d))$ and $\partial_\mu U(\mu,y)$ is continuously differentiable w.r.t. $y$ for any $(\mu,y)\in\mathcal{P}_2(\mathbb{R}^d)\times\mathbb{R}^d$, we say that $U$ is of class $\mathscr{C}^{(1,1)}(\mathcal{P}_2(\mathbb{R}^d))$, and similarly for the notation $\mathscr{C}^{(1,1)}_b(\mathcal{P}_2(\mathbb{R}^d))$.
\end{definition}
\begin{proposition}(\cite{cardaliaguet2019master}, Proposition 2.2.3)
Assume that $U$ is linear differentiable. Moreover, $\partial_y\frac{\delta U}{\delta\mu}(\mu,y)$ exists and is jointly continuous and bounded on $\mathcal{P}_2(\mathbb{R}^d)\times\mathbb{R}^d$. Then for any Borel measurable map $\phi:\mathbb{R}^d\to\mathbb{R}^d$ 
with at most linear growth, the map $s\to U((id_{\mathbb{R}^d}+s\phi)\sharp \mu)$ is differentiable at 0 and
\begin{eqnarray*}
\frac{d}{ds}U((id_{\mathbb{R}^d}+s\phi)\sharp \mu)_{|_{s=0}}=\int_{\mathbb{R}^d} \,\partial_y\frac{\delta U}{\delta\mu}(\mu,y)\phi(y)\mu(dy) ,
\end{eqnarray*}
where $id$ denotes the identity map.
\end{proposition}
It follows from Theorem 1.20 in \cite{cardaliaguet2010notes} that, under suitable regularity, the Lions derivative and the linear derivative satisfy the following relation:
\begin{eqnarray*}
\partial_{\mu}U(\mu,y)=\partial_y\frac{\delta U}{\delta\mu}(\mu,y).
\end{eqnarray*}

%Now we give some spaces of the measure variable functions.
%\begin{definition}
%We say that $U:\mathcal{P}_2(\mathbb{R}^d)\to\mathbb{R}$ belongs to $\mathscr{C}^1_b(\mathcal{P}_2(\mathbb{R}^d))$, if $U$ is Lions differentiable and $\partial_\mu U(\cdot,\cdot):\mathcal{P}_2(\mathbb{R}^d)\times\mathbb{R}^d\to\mathbb{R}^d$ is bounded and Lipschitz continuous, i.e. there exists some positive constant $L$ such that\\
%(1) $|\partial_\mu U(\mu,y)|\leq L$, $\mu\in\mathcal{P}_2(\mathbb{R}^d),\ y\in\mathbb{R}^d $,\\
%(2) $|\partial_\mu U(\mu,y)-\partial_\mu U(\mu',y')|\leq L(W_2(\mu,\mu')+|y-y'|)$, $\mu,\mu'\in\mathcal{P}_2(\mathbb{R}^d),\ y,y'\in\mathbb{R}^d$.
%\end{definition}
%\begin{definition}
%We say that $U:\mathcal{P}_2(\mathbb{R}^d)\to\mathbb{R}$ belongs to $\mathscr{C}^{1,1}_b(\mathcal{P}_2(\mathbb{R}^d))$, if $U\in\mathscr{C}^1_b(\mathcal{P}_2(\mathbb{R}^d))$ with $\partial_\mu U(\mu,\cdot):\mathbb{R}^d\to\mathbb{R}^d$ being differentiable, for any $\mu\in\mathcal{P}_2(\mathbb{R}^d)$, and the derivative $\partial_y\partial_\mu U:\mathcal{P}_2(\mathbb{R}^d)\times\mathbb{R}^d\to\mathbb{R}^d\otimes\mathbb{R}^d$ is bounded and Lipschitz continuous.
%\end{definition}

\subsection{Poisson random measures}
Poisson random measures play a fundamental role in the study of L{\'e}vy processes, with several equivalent definitions available. We follow the one given in \cite{rudiger2004stochastic}.\\
Let $E$ be a separable Banach space and $\mathcal{B}(E)$ its Borel $\sigma$-algebra. 
\begin{definition}
A process $(P_t)_{t\geq 0}$ taking values in $(E,\mathcal{B}(E))$ is called an $\mathbb{F}$-L{\'e}vy process on $(\Omega,\mathscr{F},\mathbb{F},\mathbb{P})$ if:
\begin{enumerate}
\item $P_0=0$, a.s.;
\item $(P_t)_{t\geq 0}$ is $\mathbb{F}$-adapted;
\item it has stationary and independent increments: for $0\leq s<t$, $P_t-P_s$ is independent of $\mathscr{F}_s$ and has the same distribution as $P_{t-s}$.
\end{enumerate}
\end{definition}
It is known (see, e.g., \cite{applebaum2009levy}) that an $\mathbb{F}$-L{\'e}vy process has a c{\`a}dl{\`a}g version. We set $P_{t-}:=\lim_{s\to t-}P_s$ and $\Delta P_t:=P_t-P_{t-}$.\\
For $A\in\mathcal{B}(E)$ such that $0\notin\overline{A}$ (where $\overline{A}$ denotes the closure of $A$), define
\begin{align*}
N(t,A):=\sum_{0<s\leq t}\mathbf{1}_{A}(\Delta P_s)=\#\{0<s\leq t:\Delta P_s\in A\},
\end{align*}
and for any $0\leq s<t$,
\begin{align*}
N((s,t],A):=N(t,A)-N(s,A).
\end{align*}
Then $N(t,A)$ is an $\mathbb{F}-$adapted counting process satisfying
\begin{align*}
\mathbb{P}(N(t,A)=k)=e^{-\nu_t(A)}\frac{(\nu_t(A))^k}{k!},\ \ \ {\text{where}}\ \ \nu_t(A):=\mathbb{E}[N(t,A)].
\end{align*}
Assume that for any $t$, $\nu_t$ is a finite measure on $(E,\mathcal{B}(E))$, i.e. $\nu_t(E)<\infty$. The finite intensity measure $\nu$ of $N$ is then defined by
\begin{align*}
\nu(A)(t-s):=\nu_t(A)-\nu_s(A)=\mathbb{E}[N((s,t],A)].
\end{align*}
Finally, the compensated Poisson random measure is defined as $\widetilde{N}(ds,de):=N(ds,de)-\nu(de)ds$.
For any function $f:\mathbb{R}_{+}\times (E\backslash \{0\})\to \mathbb{R}$ that is $\mathcal{B}(\mathbb{R}_{+}\times (E\backslash \{0\}))\backslash \mathcal{B}(\mathbb{R})$ measurable, 
the stochastic integral with respect to $\widetilde{N}$ is defined in the natural way: first for simple functions and then extended to general ones via approximation. For precise details, we refer to Definition 3.1 in \cite{rudiger2004stochastic}.

\subsection{Mean-field games with jump diffusion}
In this subsection, we present the master equation arising from mean-field game problems driven by jump-diffusion processes. For $t\in[0,T]$, denote by $\mathcal{A}_t$ the set of square integrable and $\mathbb{F}$-adapted c{\`a}dl{\`a}g  controls $\alpha:\Omega\times[t,T]\to\mathbb{R}$. Given  $\xi\in L^2(\Omega,\mathscr{F}_t,\mathbb{P};\mathbb{R}^d)$ and $\alpha\in\mathcal{A}_t$, consider the SDE:
\begin{align*}
X^{t,\xi,\alpha}_s=&\,\xi+\int_t^s b(r,X^{t,\xi,\alpha}_r,\mathbb{P}_{X^{t,\xi,\alpha}_r},\alpha_r)dr+\int_t^s\sigma(r,X^{t,\xi,\alpha}_r,\mathbb{P}_{X^{t,\xi,\alpha}_r})dW_r\\
&+\int_t^s\int_Eh(r,X^{t,\xi,\alpha}_{r-},\mathbb{P}_{X^{t,\xi,\alpha}_r},\theta)N(dr,d\theta),
\end{align*}
where $W$ is a $d$-dimension Brownian motion and $N$ is a Poisson random measure independent of $W$. The coefficients $b:[0,T]\times\mathbb{R}^d\times\mathcal{P}_2(\mathbb{R}^d)\times\mathbb{R}\to\mathbb{R}^d$, $\sigma:[0,T]\times\mathbb{R}^d\times\mathcal{P}_2(\mathbb{R}^d)\to\mathbb{R}^{d\times d}$ and $h:[0,T]\times\mathbb{R}^d\times\mathcal{P}_2(\mathbb{R}^d)\times E\to\mathbb{R}^d$ are measurable in all variables. 
% Moreover, $b$ is $\mathbb{F}$-adapted, $\sigma$ is $\mathbb{F}$-progressively measurable and $h$ is $\mathbb{F}-$predictable. 
%For a single player, the state follows the following equation:
%\begin{align*}
%X^{t,x,\xi,\alpha'}_s=&\xi+\int_t^s\, b(r,X^{t,x,\xi,\alpha'}_r,\mathbb{P}_{X^{t,\xi,\alpha}_r},\alpha_r')dr+\int_t^s\, \sigma(r,X^{t,x,\xi,\alpha'}_r,\mathbb{P}_{X^{t,\xi,\alpha}_r})dW_r\\
%&+\int_t^s\,\int_E\,h(r,X^{t,x,\xi,\alpha'}_{r-},\mathbb{P}_{X^{t,\xi,\alpha}_r},\theta)N(dr,d\theta),
%\end{align*}
Consider the following utility for the mean-field game:
\begin{align*}
J(t,x,\xi;\alpha,\alpha'):=\mathbb{E}\bigg[ g(X_T^{t,\xi,\alpha'},\mathbb{P}_{X_T^{t,\xi,\alpha}})+\int_t^Tf(s,X_s^{t,\xi,\alpha'},\mathbb{P}_{X_s^{t,\xi,\alpha}},\alpha'_s)ds \bigg|X_t^{t,\xi,\alpha'}=x\bigg],
\end{align*}
where $g:\mathbb{R}^d\times\mathcal{P}_2(\mathbb{R}^d)\to\mathbb{R}$ and $f:[0,T]\times\mathbb{R}^d\times\mathcal{P}_2(\mathbb{R}^d)\times\mathbb{R}\to\mathbb{R}$ are measurable in all variables.   %Moreover, $g$ is $\mathscr{F}_T-$measurable and $f$ is $\mathbb{F}-$adapted. 
When $\xi\in L^2(\Omega,\mathscr{F}_t,\mathbb{P};\mathbb{R}^d)$, it is clear that $J(t,x,\xi;\alpha,\alpha')$ is deterministic and law invariant. Thus, we can define
\begin{align*}
J(t,x,\mu;\alpha,\alpha'):=J(t,x,\xi;\alpha,\alpha'),\ \ \ \mbox{with}\ \mu=\mathbb{P}_\xi,
\end{align*}
and
\begin{align}\label{eq MFG optimal problem}
V(t,x,\mu;\alpha):=\sup_{\alpha'\in\mathcal{A}_t}J(t,x,\mu;\alpha,\alpha').
\end{align}
\begin{definition}
We say $\alpha^*\in\mathcal{A}_t$ is a mean-field equilibrium (MFE) of \eqref{eq MFG optimal problem} at $(t,\mu)$ if 
\begin{align*}
V(t,x,\mu;\alpha^*)=J(t,x,\mu;\alpha^*,\alpha^*),\ \ \ \mbox{for}\ \mu-a.e.\ \ x\in\mathbb{R}^d.
\end{align*}
\end{definition}
We remark that an MFE depends on $(t,\mu)$, but is valid for $\mu$-almost every $x$. Assuming that there is a unique MFE for each $(t,\mu)$, denoted as $\alpha^*(t,\mu)$, then the value function is
\begin{align*}
V(t,x,\mu):=V(t,x,\mu;\alpha^*(t,\mu)).
\end{align*}
Introduce the Hamiltonian function $H$ as follows
\begin{align*}
H(t,x,\mu,z):=\sup_{a\in\mathbb{R}}[b(t,x,\mu,a)\cdot z+f(t,x,\mu,a)].
\end{align*}
By the dynamic programming principle (Section 4.2 in \cite{carmona2014master}) and the It{\^o} formula for flows of measures on semimartingales (Corollary 3.5 in \cite{guo2023ito}), the value function $V$ is associated with following master equation
\begin{align*}
0=&\,\partial_tV(t,x,\mu)+\frac{1}{2}\partial^2_x V(t,x,\mu):\sigma\sigma^{\intercal}(t,x,\mu)+H(t,x,\mu,\partial_xV(t,x,\mu))\\
&+\int_E(V(t,x+h(t,x,\mu,\theta),\mu)-V(t,x,\mu))\nu(d\theta)\\
&+\int_{\mathbb{R}^d}\bigg[\frac{1}{2}\partial_y\partial_\mu V(t,x,\mu,y):\sigma\sigma^{\intercal}(t,y,\mu)+\partial_\mu V(t,x,\mu,y)\cdot \partial_p H(t,y,\mu,\partial_x V(t,y,\mu))\\
&\qquad+\int_E\Big(\frac{\delta V}{\delta\mu}(t,x,\mu,y+h(t,y,\mu,\theta))-\frac{\delta V}{\delta\mu}(t,x,\mu,y)\Big)\nu(d\theta)\bigg]\mu(dy),
\end{align*}
with $V(T,x,\mu)=g(x,\mu)$. It is well-known that $b(t,x,\mu,\alpha^*(t,\mu))=\partial_p H(t,x,\mu,\partial_x V(t,x,\mu))$ and $H(t,x,\mu,\partial_xV(t,x,\mu))=b(t,x,\mu,\alpha^*(t,\mu))\partial_xV(t,x,\mu)+f(t,x,\mu,\alpha^*(t,\mu))$. Therefore, we can write 
\begin{align*}
H(t,x,\mu,\partial_xV(t,x,\mu))=&\,\tilde{b}(t,x,\mu,\partial_xV(t,x,\mu)\sigma(t,x,\mu))\partial_xV(t,x,\mu)\\&+\tilde{f}(t,x,\mu,\partial
_xV(t,x,\mu)\sigma(t,x,\mu)),
\end{align*}
where
\begin{align*}
\tilde{b}(t,x,\mu,\partial_xV(t,x,\mu)\sigma(t,x,\mu)):=b(t,x,\mu,\alpha^*(t,\mu)),
\end{align*}
and
\begin{align*}
\tilde{f}&(t,x,\mu,\partial
_xV(t,x,\mu)\sigma(t,x,\mu))
:=f(t,x,\mu,\alpha^*(t,\mu)).
\end{align*}
Here, $\partial_pH$ denotes the derivative with respect to $z$. It would therefore be natural to denote it as $\partial_z H$; however, we retain the standard PDE notation $\partial_pH$.\\
%\begin{remark}
%This kind of master equation has a nonlinear term of the derivative of measure variable $\frac{\delta V}{\delta\mu}$. Therefore, the existence of the classical solution is not trivial, and we need to consider the viscosity solution in many cases.
%\end{remark}
In the remainder of this paper, we will mainly focus on the following more general master equation:
\begin{align}\label{eq master introduction game}
&\,\partial_tV(t,x,\mu)+\partial_x V(t,x,\mu)b(t,x,\mu,V(t,x,\mu),\partial_xV(t,x,\mu)\sigma(t,x,\mu,V(t,x,\mu)))\nonumber\\
&+f(t,x,\mu,V(t,x,\mu),\partial_xV(t,x,\mu)\sigma(t,x,\mu,V(t,x,\mu)))
\nonumber\\&+\int_E\, (V(t,x+h(t,x,\mu,V(t,x,\mu),\theta),\mu)-V(t,x,\mu))\nu(d\theta)\nonumber\\
&+\frac{1}{2}\partial_x^2V(t,x,\mu):\sigma\sigma^{\intercal}(t,x,\mu,V(t,x,\mu))\nonumber\\
&+\int_{\mathbb{R}^d}\, \bigg[\partial_\mu V(t,x,\mu,y)b(t,y,\mu,V(t,y,\mu),\partial_xV(t,y,\mu)\sigma(t,y,\mu,V(t,y,\mu)))\nonumber\\
&\qquad+\frac{1}{2}\partial_y\partial_\mu V(t,x,\mu,y):\sigma\sigma^{\intercal}(t,y,\mu,V(t,y,\mu))\nonumber\\
&\qquad+\int_E\Big(\frac{\delta V}{\delta\mu}(t,x,\mu,y+h(t,y,\mu,V(t,y,\mu),\theta))-\frac{\delta V}{\delta\mu}(t,x,\mu,y)\Big)\nu(d\theta)\bigg]\mu(dy)=0,
\end{align}
with $V(T,x,\mu)=g(x,\mu)$, where the coefficients $\sigma$ and $h$ are allowed to involve an extra nonlinear term $V(t,x,\mu)$.

\section{The well-posedness of MV-FBSDEs}
To give a probabilistic interpretation of the master equation, we introduce the following coupled MV-FBSDEs:
\begin{equation}\label{eq MVFBSDE fully 1}
\left\{\begin{split}
X^{t,\xi}_s=&\,\xi+\int_t^sb(r,X_r^{t,\xi},\mu_r^{t,\xi},Y_r^{t,\xi},Z_r^{t,\xi})dr+\int_t^s \sigma(r,X_r^{t,\xi},\mu_r^{t,\xi},Y_r^{t,\xi})dW_r\\
&+\int_t^s\int_Eh(r,X^{t,\xi}_{r-},\mu^{t,\xi}_r,Y^{t,\xi}_{r-},\theta)N(dr,d\theta),\\
Y^{t,\xi}_s=&\,g(X^{t,\xi}_T,\mu_T^{t,\xi})+\int_s^T f(r,X_r^{t,\xi},\mu_r^{t,\xi},Y_r^{t,\xi},Z_r^{t,\xi})dr-\int_s^T Z_r^{t,\xi}dW_r\\
&-\int_t^T\int_E H^{t,\xi}_r(\theta)\tilde{N}(dr,d\theta),
\end{split}\right.
\end{equation}
and 
\begin{equation}\label{eq MVFBSDE fully 2}
\left\{\begin{split}
X_s^{t,x,\mu}=&\,x+\int_t^s b(r,X_r^{t,x,\mu},\mu_r^{t,\xi},Y_r^{t,x,\mu},Z_r^{t,x,\mu})dr+\int_t^s\sigma(r,X_r^{t,x,\mu},\mu_r^{t,\xi},Y_r^{t,x,\mu})dW_r\\
&+\int_t^s\int_E h(r,X^{t,x,\mu}_{r-},\mu^{t,\xi}_r,Y^{t,x,\mu}_{r-},\theta)N(dr,d\theta),\\
Y_s^{t,x,\mu}=&\,g(X^{t,x,\mu}_T,\mu_T^{t,\xi})+\int_s^T f(r,X_r^{t,x,\mu},\mu_r^{t,\xi},Y_r^{t,x,\mu},Z_r^{t,x,\mu})dr-\int_s^TZ_r^{t,x,\mu}dW_r\\
&-\int_t^T\int_E H^{t,x,\mu}_r(\theta)\tilde{N}(dr,d\theta),
\end{split}\right.
\end{equation}
where $\mu_r^{t,\xi}=\mathbb{P}_{X_r^{t,\xi}}$ 
%for any {\red $\xi\in L^2(\Omega,\mathscr{F}_t,\mathbb{P};\mathbb{R}^d)$} 
and $\tilde{N}(dr,d\theta):=N(dr,d\theta)-\nu(d\theta)dr$.
The coefficients  $b:[0,T]\times\mathbb{R}^d\times\mathcal{P}_2(\mathbb{R}^d)\times\mathbb{R}\times\mathbb{R}^d\to\mathbb{R}^d$, $\sigma:[0,T]\times\mathbb{R}^d\times\mathcal{P}_2(\mathbb{R}^d)\times\mathbb{R}\to\mathbb{R}^{d\times d}$, 
$h:[0,T]\times\mathbb{R}^d\times\mathcal{P}_2(\mathbb{R}^d)\times\mathbb{R}\times E\to\mathbb{R}^d$, $f:[0,T]\times\mathbb{R}^d\times\mathcal{P}_2(\mathbb{R}^d)\times\mathbb{R}\times\mathbb{R}^d\to\mathbb{R}$, $g:\mathbb{R}^d\times\mathcal{P}_2(\mathbb{R}^d)\to\mathbb{R}$ are measurable in all variables. %Moreover, $b,f$ are $\mathbb{F}-$adapted, $\sigma$ is $\mathbb{F}-$progressively measurable, $h$ is $\mathbb{F}-$predictable and $g$ is $\mathscr{F}_T-$measurable.

\subsection{Assumptions}
We impose the following assumptions to ensure the well-posedness of the MV-SDE.
\begin{assumption}\label{assumption SDE}
(1) The initial value $X_0\in L^2(\Omega,\mathscr{F}_0,\mathbb{P};\mathbb{R}^d)$. \\
(2) (Lipschitz Continuity) There exists a constant $L>0$ such that for any $t\in[0,T],y\in\mathbb{R},z\in\mathbb{R}^{d}$ and $x_1,x_2\in\mathbb{R}^d,\mu_1,\mu_2\in\mathcal{P}_2(\mathbb{R}^d)$, 
\begin{align*}
|b(t,x_1,\mu_1,y,z)-b(t,x_2,\mu_2,y,z)|&\leq L(|x_1-x_2|+W_2(\mu_1,\mu_2)), \\
\|\sigma(t,x_1,\mu_1,y)-\sigma(t,x_2,\mu_2,y)\|&\leq L(|x_1-x_2|+W_2(\mu_1,\mu_2)).
\end{align*}
Moreover, there exists a measurable function $L(\theta)>0$ on $E$ such that
\begin{align*}
|h(t,x_1,\mu_1,y,\theta)-h(t,x_2,\mu_2,y,\theta)|&\leq L(\theta)(|x_1-x_2|+W_2(\mu_1,\mu_2)),
\end{align*}
and for some constant $\gamma>0$,
\begin{align}\label{eq Ltheta}
\int_E\,(e^{\gamma L(\theta)}-1)\nu(d\theta)\leq L.
\end{align}
(3) (Boundedness) There is a constant $M>0$, such that for any $t\in[0,T],y\in\mathbb{R},z\in\mathbb{R}^{d}$,
\begin{align*}
|b(t,0,\delta_0,y,z)|+\|\sigma(t,0,\delta_0,y)\|+\int_E\,(e^{\gamma|h(t,0,\delta_0,y,\theta)|}-1)\nu(d\theta)\leq M,
\end{align*}
where $\delta_0$ denotes the Dirac measure at $0$.
\end{assumption}
This assumption differs from that in \cite{li2018mean} by relaxing the requirement that the growth and Lipschitz conditions on $h$ be uniform in $\theta$. This is made possible by our framework's use of a finite measure $\nu$, which allows the $\theta$-dependence to be controlled via an integrated condition like (\ref{eq Ltheta}). \\
The following lemma concerns the well-posedness of MV-SDEs.
\begin{lemma}\label{lemma SDE wellposed}
Let Assumption \ref{assumption SDE} hold. For any $(Y,Z)\in\mathcal{S}^2_{\mathbb{F}}(0,T;\mathbb{R})\times\mathcal{M}^2_\mathbb{F}(0,T;\mathbb{R}^d)$, the following MV-SDE
\begin{align*}
X_t=&\,X_0+\int_0^tb(s,X_s,\mu_s,Y_s,Z_s)ds+\int_0^t \sigma(s,X_s,\mu_s,Y_s)dW_s\\
&+\int_0^t\int_Eh(s,X_{s-},\mu_s,Y_{s-},\theta)N(ds,d\theta)
\end{align*}
has a unique solution $X\in\mathcal{S}^2_\mathbb{F}(0,T;\mathbb{R}^d)$, where $\mu_t=\mathbb{P}_{X_t}$ for any $t$.
\end{lemma}
The proof of this lemma follows a standard Picard iteration argument, analogous to that for the classical jump-diffusion SDEs (see Theorem 6.2.3 in \cite{applebaum2009levy}). We therefore omit the details and proceed to study the well-posedness of the following MV-FBSDE system, which is equivalent to \eqref{eq MVFBSDE fully 1},
\begin{equation}\label{eq MV-FBSDE}
\left\{\begin{split}
X_t=&\,X_0+\int_0^tb(r,X_r,\mu_r,Y_r,Z_r)dr+\int_0^t \sigma(r,X_r,\mu_r,Y_r)dW_r\\
&\qquad +\int_0^t\int_Eh(r,X_{r-},\mu_r,Y_{r-},\theta)N(dr,d\theta),\\
Y_t=&\,g(X_T,\mu_T)+\int_t^Tf(r,X_r,\mu_r,Y_r,Z_r)dr-\int_t^TZ_rdW_r-\int_t^T\int_EH_r(\theta)\tilde{N}(dr,d\theta).
\end{split}\right.
\end{equation}
For a solution quadruple $\Theta:=(X,Y,Z,H)$ to \eqref{eq MV-FBSDE}, we define its norm as
\begin{align*}
\|\Theta\|:=\mathbb{E}\bigg[\sup_{t\in[0,T]}|X_t|^2+\sup_{t\in[0,T]}|Y_t|^2+\int_0^T |Z_t|^2dt+\int_0^T\int_E |H_t(\theta)|^2\nu(d\theta)dt\bigg].
\end{align*}
To establish the well-posedness of the MV-FBSDE \eqref{eq MV-FBSDE}, we make the following assumptions.
\begin{assumption}\label{assumption MVFBSDE f}
$f$ and $g$ are uniformly Lipschitz continuous in $(x,\mu)$ and are bounded, i.e. for any $t\in[0,T],y\in\mathbb{R},z\in\mathbb{R}^{d}$, $x_1,x_2\in\mathbb{R}^d$ and $\mu_1,\mu_2\in\mathcal{P}_2(\mathbb{R}^d)$, there exist two positive constants $L$ and $M$ such that,
%There exist two positive constants $L$ and $M$ such that, for any $t\in[0,T],y\in\mathbb{R},z\in\mathbb{R}^{d}$, $x_1,x_2\in\mathbb{R}^d$ and $\mu_1,\mu_2\in\mathcal{P}_2(\mathbb{R}^d)$, 
\begin{align*}
|f(t,x_1,\mu_1,y,z)-f(t,x_2,\mu_2,y,z)|&\leq L(|x_1-x_2|+W_2(\mu_1,\mu_2)),\\
|g(x_1,\mu_1)-g(x_2,\mu_2)|&\leq L(|x_1-x_2|+ W_2(\mu_1,\mu_2))
\end{align*}
and 
\begin{align*}
|f(t,0,\delta_0,y,z)|+|g(0,\delta_0)|\leq M.
\end{align*}
\end{assumption}
\begin{assumption}\label{assumption y lip}
The coefficients $\Sigma:=b,\sigma,f$ are uniformly Lipschitz continuous in $(y,z)$, i.e. for any $t\in[0,T]$, $x\in\mathbb{R}^d$, $\mu\in\mathcal{P}_2(\mathbb{R}^d)$, $y_1,y_2\in\mathbb{R}$ and $z_1,z_2\in\mathbb{R}^d$, there is a constant $L>0$ such that
\begin{align*}
|\Sigma(t,x,\mu,y_1,z_1)-\Sigma(t,x,\mu,y_2,z_2)|\leq L(|y_1-y_2|+|z_1-z_2|).
\end{align*}
$h$ is locally Lipschitz continuous in $y$, i.e. for any $t\in[0,T]$, $x\in\mathbb{R}^d$, $\mu\in\mathcal{P}_2(\mathbb{R}^d)$, $\theta\in E$ and $y_1,y_2\in\mathbb{R}$, there exists a measurable function $L(\theta)>0$ on $E$ such that
\begin{align*}
|h(t,x,\mu,y_1,\theta)-h(t,x,\mu,y_2,\theta)|\leq L(\theta)|y_1-y_2|,
\end{align*}
where $L(\theta)$ satisfies the integrability condition \eqref{eq Ltheta}.
\end{assumption}

\subsection{The well-posedness of MV-FBSDEs on small time duration}

\begin{theorem}\label{thm FBSDE small time duration}
Let Assumptions \ref{assumption SDE}-\ref{assumption y lip} hold. Then there exists a constant $\delta(L)>0$, depending only on the constants in the assumptions, such that for any $T\leq\delta(L)$, the MV-FBSDE \eqref{eq MV-FBSDE} has a unique solution. Moreover, the following estimate holds 
\begin{align}\label{eq theta estimate}
\|\Theta\|\leq C(\mathbb{E}|X_0|^2+I^2_0),
\end{align}
where $C$ is a constant depending on $L$ and $T$, and 
\begin{align*}
I^2_0:=&\int_0^T \Big(|b(t,0,\delta_0,0,0)|^2+\|\sigma(t,0,\delta_0,0)\|^2+|f(t,0,\delta_0,0,0,0)|^2\\&\qquad\quad+\int_E|h(t,0,\delta_0,0,\theta)|^2d\nu(\theta)\Big)dt+|g(0,\delta_0)|^2.
\end{align*}
\end{theorem}
\begin{proof}
For $(y,z)\in \mathcal{S}^2_\mathbb{F}(0,T;\mathbb{R}) \times \mathcal{M}^2_\mathbb{F}(0,T;\mathbb{R}^d)$, define a map $F$ by $F(y,z):=(Y^{y,z},Z^{y,z})$, where $\Theta^{y,z}:=(X^{y,z},Y^{y,z},Z^{y,z},H^{y,z})$ is the unique solution to the following decoupled FBSDE:
\begin{equation}\label{eq MVFBSDE decoupled}
\left\{\begin{split}
X_t^{y,z}=&\,X_0+\int_0^tb(r,X_r^{y,z},\mu^{y,z}_r,y_r,z_r)dr+\int_0^t \sigma(r,X^{y,z}_r,\mu^{y,z}_r,y_r)dW_r\\
&+\int_0^t\int_Eh(r,X^{y,z}_{r-},\mu_r^{y,z},y_{r-},\theta)N(dr,d\theta),\\
Y^{y,z}_t=&\,g(X^{y,z}_T,\mu^{y,z}_T)+\int_t^T f(r,X_r^{y,z},\mu^{y,z}_r,Y^{y,z}_r,Z^{y,z}_r)dr-\int_0^tZ^{y,z}_rdW_r\\
&-\int_0^t\int_E H^{y,z}_r(\theta)\tilde{N}(dr,d\theta),
\end{split}\right.
\end{equation}
where $\mu^{y,z}_t:=\mathbb{P}_{X^{y,z}_t}$. The well-posedness of the above system is clear: we first solve the SDE by Lemma \ref{lemma SDE wellposed}, then obtain $(Y^{y,z}, Z^{y,z}, H^{y,z})$ via the BSDE theory in \cite{li2018mean}. We'll show that $F$ is a contraction mapping under the norm:
\begin{align*}
\|(y,z)\|:=\mathbb{E}\bigg[\sup_{t\in[0,T]}|y_t|^2+\int_0^T |z_t|^2dt\bigg].
\end{align*}
Let $(y^i, z^i)\in \mathcal{S}^2_\mathbb{F}(0,T;\mathbb{R}) \times \mathcal{M}^2_\mathbb{F}(0,T;\mathbb{R}^d)$ for $i=1,2$, and let $\Theta^i := (X^i, Y^i, Z^i, H^i)$ be the corresponding solutions of \eqref{eq MVFBSDE decoupled}. Define $\Delta Y_t:=Y^1_t-Y^2_t,\Delta
 Z_t:=Z^1_t-Z^2_t$, $\Delta X_t:=X^1_t-X^2_t,\Delta \mu_t:=W_2(\mu_t^1,\mu_t^2)=W_2(\mathbb{P}_{X^1_t},\mathbb{P}_{X^2_t}),\Delta H_t:=H^1_t-H^2_t$, for all $t\in[0,T]$. Then
\begin{align*}
\Delta Y_t=&\,g(X_T^1,\mu^1_T)-g(X_T^2,\mu^2_T)+\int_t^T \big[f(r,X^1_r,\mu^1_r,Y^1_r,Z^1_r)-f(r,X^2_r,\mu^2_r,Y^2_r,Z^2_r)\big]dr\\
&-\int_t^T \Delta Z_rdW_r-\int_t^T\int_E\Delta H_r(\theta)\tilde{N}(dr,d\theta)\\
=&\,\lambda^1\Delta X_T+\lambda^2\Delta\mu_T+\int_t^T\big[\alpha^3_r\Delta X_r+\beta^3_r\Delta\mu_r+\gamma^3_r\Delta Y_r+\rho^3_r \Delta Z_r\big]dr\\
&-\int_t^T \Delta Z_rdW_r-\int_t^T\int_E\Delta H_r(\theta)\tilde{N}(dr,d\theta),
\end{align*}
where 
\begin{align*}
\lambda^1:=&\frac{g(X_T^1,\mu^1_T)-g(X^2_T,\mu^1_T)}{X^1_T-X^2_T},\ \ \lambda^2:=\frac{g(X_T^2,\mu^1_T)-g(X^2_T,\mu^2_T)}{W_2(\mu_T^1,\mu^2_T)},\\
\alpha^3_r:=&\frac{f(r,X^1_r,\mu^1_r,y^1_r,z^1_r)-f(r,X^2_r,\mu^1_r,Y^1_r,Z^1_r)}{X^1_r-X^2_r},\\
\beta^3_r:=&\frac{f(r,X^2_r,\mu^2_r,Y^1_r,Z^1_r)-f(r,X^2_r,\mu^1_r,Y^1_r,Z^1_r)}{W_2(\mu_r^1,\mu^2_r)},\\
\gamma^3_r:=&\frac{f(r,X^2_r,\mu^2_r,Y^2_r,Z^1_r)-f(r,X^2_r,\mu^2_r,Y^1_r,Z^1_r)}{Y^1_r-Y^2_r},\\
\rho^3_r:=&\frac{f(r,X^2_r,\mu^2_r,Y^2_r,Z^2_r)-f(r,X^2_r,\mu^2_r,Y^2_r,Z^1_r)}{Z^1_r-Z^2_r},
\end{align*}
and by Assumptions \ref{assumption MVFBSDE f} and \ref{assumption y lip}, we know that all the above coefficients are bounded by $L$.\\
With the similar notations, we have
\begin{align*}
\Delta X_t=&\,\int_0^t\big[\alpha_r^1\Delta X_r+\beta_r^1\Delta\mu_r+\gamma^1_r\Delta y_r+\rho^1_r\Delta z_r\big]dr+\int_0^t\big[\alpha_r^2\Delta X_r+\beta_r^2\Delta\mu_r+\gamma^2_r\Delta y_r\big]dW_r\\
&+\int_0^t\int_E\big[\alpha_r^4(\theta)\Delta X_{r-}+\beta_r^4(\theta)\Delta \mu_r+\gamma^4_r(\theta)\Delta y_{r-}\big]N(dr,d\theta).
\end{align*}
Then, noting that 
\begin{align*}
\Delta \mu_t=W_2(\mu_t^1,\mu_t^2)\leq \left[\mathbb{E}|X^1_t-X^2_t|^2\right]^{\frac{1}{2}},
\end{align*}
and applying It\^o's formula for jump-diffusion processes (see Lemma 4.4.6 in \cite{applebaum2009levy}), we obtain
\begin{align}\label{eq small time x}
&\mathbb{E}|\Delta X_t|^2\nonumber\\
=&\,\mathbb{E}\bigg[\int_0^t2\Delta X_r(\alpha_r^1\Delta X_r+\beta_r^1\Delta\mu_r+\gamma^1_r\Delta y_r+\rho^1_r\Delta z_r)dr+\int_0^t(\alpha_r^2\Delta X_r+\beta_r^2\Delta\mu_r+\gamma^2_r\Delta y_r)^2dr\nonumber\\
&\qquad+\int_0^t\int_E\big[(\Delta X_{r-}+\alpha_r^4(\theta)\Delta X_{r-}+\beta_r^4(\theta)\Delta \mu_r+\gamma^4_r(\theta)\Delta y_{r-})^2-(\Delta X_{r-})^2\big]N(dr,d\theta)\bigg]\nonumber\\
\leq & \,\delta(L)\Big(C+\frac{4\|\gamma^1\|^2_{\infty}+4\|\gamma^4\|^2_{\infty}+4\|\rho^1\|^2_\infty}{ \eps}\Big)\sup_{t\in[0,T]}\mathbb{E}|\Delta X_t|^2\nonumber\\&+\big(3 \eps+\delta(L)\|\gamma^2\|_{\infty}^2+\delta(L)\|\gamma^4\|_{\infty}^2\big)\|(\Delta y,\Delta z)\|^2,
\end{align}
where $\|\gamma\|_{\infty}:=\mathbb{E}\Big[\sup_{t\in[0,T]}\int_E|\gamma_t(\theta)|^2\nu(d\theta)\Big]^{\frac{1}{2}}\leq L$.
Still applying It\^o's formula to $|\Delta Y_t|^2$ yields
\begin{equation}\label{eq small time y}\begin{split}
&\mathbb{E}|\Delta Y_t|^2+\mathbb{E}\int_t^T|\Delta Z_r|^2dr+\mathbb{E}\int_t^T\int_E|\Delta H_r(\theta)|^2\nu(d\theta)dr\\
=&\,\mathbb{E}\bigg[(\lambda^1\Delta X_T+\lambda^2\Delta\mu_T)^2+\int_t^T2\Delta Y_r(\alpha^3_r\Delta X_r+\beta^3_r\Delta\mu_r+\gamma^3_r\Delta Y_r+\rho^3_r\Delta Z_r) dr\bigg]\\
\leq & \,(|\lambda^1|+|\lambda^2|+C\delta(L))\sup_{t\in[0,T]}\mathbb{E}|\Delta X_t|^2+\mathbb{E}\int_t^TC\eps^{-1}|\Delta Y_r|^2+\eps|\Delta Z_r|^2\,dr\\
\leq & \,(|\lambda^1|+|\lambda^2|+C\delta(L))\sup_{t\in[0,T]}\mathbb{E}|\Delta X_t|^2+(C\eps^{-1}\delta(L)+\eps)\|(\Delta Y,\Delta Z)\|^2.
\end{split}\end{equation}
Combining \eqref{eq small time x} and \eqref{eq small time y} provides that
\begin{align*}
&\|(\Delta Y,\Delta Z)\|\\\leq&\, \bigg[\frac{|\lambda^1|+|\lambda^2|+C\delta(L)}{1-C\eps^{-1}\delta(L)-\eps}\bigg]\bigg[\frac{3\eps+\delta(L)\|\gamma^2\|_{\infty}^2+\delta(L)\|\gamma^4\|_{\infty}^2}{1-\delta(L)(C+\frac{4\|\gamma^1\|^2_{\infty}+4\|\gamma^4\|^2_{\infty}+\|\rho^1\|^2_\infty}{\eps})}C(1+\delta(L))\bigg]\|(\Delta y,\Delta z)\|.
\end{align*}
By Assumption \ref{assumption y lip}, the quantities $\gamma^i$ are bounded, and the constant $C$ depends only on the Lipschitz constants. Hence, we can choose $\varepsilon$ and $\delta(L)$ small enough such that
\begin{align*}
\|(\Delta Y,\Delta Z)\|\leq \theta'\|(\Delta y,\Delta z)\|,\ \ \ \mbox{with}\ \theta'<1.
\end{align*}
This proves that $F$ is a contraction mapping.
Let $(Y,Z)$ be the fixed point of $F$ (i.e., $F(Y,Z) = (Y,Z)$), then by the definition of $F$, we have
\begin{equation*}
\left\{\begin{split}
    &X_t^{Y,Z}=X_0+\int_0^t b(r,X_r^{Y,Z},\mu^{Y,Z}_r,Y_r,Z_r)dr+\int_0^t \sigma(r,X^{Y,Z}_r,\mu^{Y,Z}_r,Y_r)dW_r\\
    &\qquad\qquad+\int_0^t\int_E h(r,X_{r-}^{Y,Z},\mu_r^{Y,Z},Y_{r-},\theta)N(dr,d\theta),\\
    &Y_t=g(X^{Y,Z}_T,\mu^{Y,Z}_T)+\int_t^T f(r,X^{Y,Z}_r,\mu^{Y,Z}_r,Y_r,Z_r)dr-\int_t^TZ_rdW_r\\
    &\qquad\qquad-\int_t^T\int_E H_r^{Y,Z}(\theta)\tilde{N}(dr,d\theta).
\end{split}\right.
\end{equation*}
Hence, the quadruple $(X^{Y,Z},Y,Z,H^{Y,Z})$ solves the MV-FBSDE system \eqref{eq MV-FBSDE}.\\
Now we prove the estimate \eqref{eq theta estimate}. Let $(X^0,Y^0,Z^0,H^0)$ be the solution to the following FBSDE
\begin{equation*}
\left\{\begin{split}
&X^0_t=X_0+\int_0^tb(r,X_r^0,\mu^0_r,0,0)dr+\int_0^t \sigma(r,X^0_r,\mu^0_r,0)dW_r\\
&\qquad\qquad+\int_0^t\int_Eh(r,X_{r-}^0,\mu_r^0,0,\theta)N(dr,d\theta),\\
&Y^0_t=g(X^0_T,\mu^0_T)+\int_t^Tf(r,X^0_r,\mu^0_r,Y^0_r,Z^0_r)dr-\int_t^T\, Z^0_rdW_r-\int_t^T\int_E H^0_r(\theta)\tilde{N}(dr,d\theta).
\end{split}\right.
\end{equation*}
By the previous proof, we have 
\begin{align*}
\|(Y,Z)\|-\|(Y^0,Z^0)\|\leq \|(Y-Y^0,Z-Z^0)\|\leq \theta'\|(Y,Z)\|.
\end{align*}
Hence $\|(Y,Z)\|\leq \frac{1}{1-\theta'}\|(Y^0,Z^0)\|$. It is straightforward to show $\|(Y^0,Z^0)\|\leq CI_0^2$, which in turn provides the estimate for $\|(Y,Z)\|$. One can then apply the a priori estimates for MV-SDEs to obtain the estimate for $X$ in \eqref{eq theta estimate} immediately. Finally, to obtain the estimate for $\mathbb{E}\int_0^T\int_E|H_t(\theta)|^2\nu(d\theta)dt$, we apply It{\^o}'s formula to $|Y_t|^2$ and derive the conclusion by using standard arguments (see, e.g., in \cite{barles1997backward}).
\end{proof}
The next proposition establishes the continuity of the solution to MV-FBSDE \eqref{eq MV-FBSDE} with respect to the initial distribution.
\begin{proposition}\label{proposition continuity}
Let Assumptions \ref{assumption SDE}-\ref{assumption y lip} hold and $T\leq \delta(L)$. For any $X^1_0,X^2_0\in L^2(\Omega,\mathscr{F}_0,\mathbb{P};\mathbb{R}^d)$, 
let $(X^1,Y^1,Z^1,H^1)$ and $(X^2,Y^2,Z^2,H^2)$ be the solutions of MV-FBSDE \eqref{eq MV-FBSDE} with initial values $X^1_0$ and $X^2_0$, respectively. Then we have
\begin{align*}
\|(Y^1-Y^2,Z^1-Z^2,H^1-H^2)\|\leq C\left[\mathbb{E}|X^1_0-X^2_0|^2\right]^{\frac{1}{2}},
\end{align*}
where $C$ is a constant depending on $L$ and $T$.
\end{proposition}
\begin{proof}
With the same notations as in the proof of Theorem \ref{thm FBSDE small time duration}, we have
\begin{equation*}
\left\{\begin{split}
\Delta X_t=&\,\Delta X_0+\int_0^t\big[\alpha_r^1\Delta X_r+\beta_r^1\Delta\mu_r+\gamma_r^1\Delta Y_r+\rho^1_r\Delta Z_r\big]dr+\int_0^t\big[\alpha_r^2\Delta X_r+\beta_r^2\Delta\mu_r\\&+\gamma_r^3\Delta Y_r\big]dW_r
+\int_0^t\int_E\big[\alpha^4_r(\theta)\Delta X_{r-}+\beta^4_r(\theta)\Delta \mu_r+\gamma^4_r(\theta)\Delta Y_{r-}\big]N(dr,d\theta),\\
\Delta Y_t=&\,\lambda^1\Delta X_T+\lambda^2 \Delta\mu_T+\int_t^T\big[\alpha^3_r\Delta X_r+\beta^3_r\Delta\mu_r+\gamma^3_r\Delta Y_r+\rho^3_r\Delta Z_r\big]dr\\&-\int_t^T\Delta Z_rdW_r-\int_t^T\int_E \Delta H_r(\theta)\tilde{N}(dr, d\theta).
\end{split}\right.\end{equation*}
Then it follows from the proof of Theorem \ref{thm FBSDE small time duration} that
%\begin{align*}
%\mathbb{E}\big[|\Delta X_t-\Delta X_0|^2\big]\leq&\,\delta\Big(C+\frac{4\|\gamma^1\|^2_{\infty}+4\|\gamma^4\|^2_{\infty}+4\|\rho^1\|^2_\infty}{\eps}\Big)\sup_{t\in[0,T]}\mathbb{E}\big[|\Delta X_t|^2\big]\\&+\big(3\eps+\delta\|\gamma^2\|_{\infty}^2+\delta\|\gamma^4\|_{\infty}^2\big)\|(\Delta Y,\Delta Z)\|^2.
%\end{align*}
%Then 
\begin{align*}
\mathbb{E}|\Delta X_t|^2&\leq \mathbb{E}|\Delta X_0|^2+ \delta(L)\Big(C+\frac{4\|\gamma^1\|^2_{\infty}+4\|\gamma^4\|^2_{\infty}+4\|\rho^1\|^2_\infty}{\eps}\Big)\sup_{t\in[0,T]}\mathbb{E}|\Delta X_t|^2\\
&\quad+\big(3\eps+\delta(L)\|\gamma^2\|_{\infty}^2+\delta(L)\|\gamma^4\|_{\infty}^2\big)\|(\Delta Y,\Delta Z)\|^2
\end{align*}
and 
\begin{align}\label{eq prop2-1}
&\|(\Delta Y,\Delta Z)\|^2+\mathbb{E}\int_t^T\int_E\Big(|\Delta Y_{r-}+\Delta H_r(\theta)|^2-|\Delta Y_{r-}|^2-2\Delta Y_{r-}\Delta H_r(\theta)\Big)\nu(d\theta)dr\nonumber\\
\leq&\,\big(|\lambda^1|+|\lambda^2|+C\delta(L)\big)\sup_{t\in[0,T]}\mathbb{E}|\Delta X_t|^2+(C\eps^{-1}\delta(L)+\eps)\|(\Delta Y,\Delta Z)\|^2.
\end{align}
Combining the above two estimates yields
\begin{align*}
\|(\Delta Y,\Delta Z)\|\leq \frac{C(1+\delta(L))}{1-\theta'}\mathbb{E}|\Delta X_0|^2.
\end{align*}
Therefore, there exists a constant $C$ such that 
\begin{align*}
\sup_{t\in[0,T]}\mathbb{E}|Y^1_t-Y^2_t|^2+\mathbb{E}\int_0^T|Z_t^1-Z_t^2|^2dt\leq C\mathbb{E}|X^1_0-X^2_0|^2.
\end{align*}
Still by \eqref{eq prop2-1}, we obtain
\begin{align*}
\mathbb{E}\int_0^T\int_E|\Delta H_r(\theta)|^2\nu(d\theta)dr
\leq&\, (|\lambda^1|+|\lambda^2|+C\delta(L))\sup_{t\in[0,T]}\mathbb{E}|\Delta X_t|^2\\&+(C\eps^{-1}\delta(L)+\eps)\|(\Delta Y,\Delta Z)\|^2,
\end{align*}
which leads us to the conclusion.
%By Proposition \ref{proposition invariant} we know that $Y^i,Z^i$ depends on $X^i_0$ only through its distribution, $i=1,2$. Hence we have
%\begin{align*}
%\|(\Delta Y,\Delta Z)\|^2\leq C\inf_{\mathbb{P}_{U}=\mathbb{P}_{X^1_0},\mathbb{P}_{V}=\mathbb{P}_{X^2_0}}\mathbb{E}[|U-V|^2]=W_2(\mathbb{P}_{X^1_0},\mathbb{P}_{X^2_0})^2.
%\end{align*}
\end{proof}
Once the well-posedness of MV-FBSDE \eqref{eq MVFBSDE fully 1} is established, one can immediately obtain the well-posedness of \eqref{eq MVFBSDE fully 2} using the classical theory of FBSDEs, since the measure $\mathbb{P}_{X^{t,\xi}}$ is no longer a variable. (See \cite{pardoux1999forward} and \cite{zhang2006wellposedness} for details.) Next, we present a proposition concerning the continuity of the solution to \eqref{eq MVFBSDE fully 2}.
\begin{proposition}\label{proposition continuity eq2}
Let the conditions of Proposition \ref{proposition continuity} hold. Given any $x,x'\in\mathbb{R}^d,\mu,\mu'\in\mathcal{P}_2(\mathbb{R}^d)$, let $(X^{t,x,\mu},Y^{t,x,\mu},Z^{t,x,\mu},H^{t,x,\mu})$ and $(X^{t,x',\mu'},Y^{t,x',\mu'},Z^{t,x',\mu'},H^{t,x',\mu'})$ be the solutions of \eqref{eq MVFBSDE fully 2} with initial values $(t,x,\mu)$ and $(t,x',\mu')$, respectively. Then we have
\begin{align*}
\|(Y^{t,x,\mu}-Y^{t,x',\mu'},Z^{t,x,\mu}-Z^{t,x',\mu'},H^{t,x,\mu}-H^{t,x',\mu'})\|\leq C(|x-x'|+W_2(\mu,\mu')).
\end{align*}
\end{proposition}
\begin{proof}
The continuity with respect to $x$ is well-known (see, e.g., Theorem 3 in \cite{zhang2006wellposedness}); hence we focus on proving continuity with respect to the measure $\mu$. To this end, define $\Delta_\mu\Theta^{t,x,\mu}:=\Theta^{t,x,\mu}-\Theta^{t,x,\mu'}$ and  $\Delta_\mu \Theta^{t,\xi}:=\Theta^{t,\xi}-\Theta^{t,\xi'}$, where $\mu:=\mathbb{P}_{\xi}$ and $\mu':=\mathbb{P}_{\xi'}$. By using the same method as in Proposition \ref{proposition continuity}, we obtain
\begin{align*}
\mathbb{E}&|\Delta_\mu X_t^{t,x,\mu}|^2\leq \, \delta(L)\Big(C+\frac{4\|\gamma^1\|^2_{\infty}+4\|\gamma^4\|^2_{\infty}+4\|\rho^1\|^2_\infty}{\eps}\Big)\sup_{t\in[0,T]}\mathbb{E}|\Delta_\mu X^{t,x,\mu}_t|^2\\&+C\delta(L)\sup_{t\in[0,T]}\mathbb{E}|\Delta_\mu X^{t,\xi}_t|^2+\big(3\eps+\delta(L)\|\gamma^2\|_{\infty}^2+\delta(L)\|\gamma^4\|_{\infty}^2\big)\|(\Delta_\mu Y^{t,x,\mu},\Delta_\mu Z^{t,x,\mu})\|^2
\end{align*}
and 
\begin{align*}
\|(\Delta_\mu Y^{t,x,\mu}&,\Delta_\mu Z^{t,x,\mu},\Delta_\mu H^{t,x,\mu})\|^2
\leq(|\lambda^1|+C\delta(L))\sup_{t\in[0,T]}\mathbb{E}|\Delta_\mu X^{t,x,\mu}_t|^2\\&+(|\lambda^2|+C\delta(L))\sup_{t\in[0,T]}\mathbb{E}|\Delta_\mu X^{t,\xi}_t|^2+(C\eps^{-1}\delta(L)+\eps)\|(\Delta_\mu Y^{t,x,\mu},\Delta_\mu Z^{t,x,\mu})\|^2.
\end{align*}
Combining the above estimates provides 
\begin{align*}
\|(\Delta_\mu Y^{t,x,\mu},\Delta_\mu Z^{t,x,\mu},\Delta_\mu H^{t,x,\mu})\|^2\leq C\mathbb{E}|\xi-\xi'|^2.
\end{align*}
Since the solution $\Theta^{t,x,\mu}$ depends on $\xi$ only through its distribution, we have
\begin{align*}
\|(\Delta_\mu Y^{t,x,\mu},\Delta_\mu Z^{t,x,\mu},\Delta_\mu H^{t,x,\mu})\|\leq C\inf_{ \xi\sim\mu,\xi'\sim\mu'}(\mathbb{E}|\xi-\xi'|^2)^{\frac{1}{2}}=CW_2(\mu,\mu'),
\end{align*}
which completes the proof.
\end{proof}

\section{The first order derivative of the solution}
To analyze the classical solutions of master equation \eqref{eq master introduction game}, it is necessary to study the differentiability of the solutions to MV-FBSDE \eqref{eq MVFBSDE fully 2} with respect to both $x$ and $\mu$. This requires the following $2p$-order moment estimates ($p \geq 1$).
\begin{theorem}\label{thm p-estimate} Let $\Theta^{t,\xi}$ and $\Theta^{t,x,\mu}$ be the solutions to the MV-FBSDEs \eqref{eq MVFBSDE fully 1} and \eqref{eq MVFBSDE fully 2}, respectively. Under Assumptions \ref{assumption SDE}--\ref{assumption y lip} and for any $T\leq \delta(L)$, there exists for any $p\geq 1$ a constant $C_{p,L}>0$, depending only on $p$ and the Lipschitz constants, such that
\begin{align*}
&\|\Theta^{t,\xi}-\Theta^{t,\xi'}\|_{2p}:=\mathbb{E}\bigg[\sup_{s\in[t,T]}|X^{t,\xi}_s-X_s^{t,\xi'}|^{2p}+\sup_{s\in[t,T]}|Y^{t,\xi}_s-Y^{t,\xi'}_s|^{2p}+\Big(\int_t^T|Z^{t,\xi}_s-Z^{t,\xi'}_s|^{2}ds\Big)^p\\
&\qquad\qquad\qquad\qquad+\Big(\int_t^T\int_E|H^{t,\xi}_s(\theta)-H^{t,\xi'}_s(\theta)|^{2}d\nu(\theta)ds\Big)^p\bigg]\leq C_{p,L}\mathbb{E}|\xi-\xi'|^{2p}\end{align*}
and
\begin{align*}
\|\Theta^{t,x,\mu}-\Theta^{t,x',\mu'}\|_{2p}\leq C_{p,L}\big(|x-x'|^{2p}+W_2(\mu,\mu')^{2p}\big).
\end{align*}
\end{theorem}
\begin{proof}
To prove the first assertion, we begin by assuming that the processes $\Theta^{t,\xi}$ and $\Theta^{t,\xi'}$ have finite moments of $2p$-order and start with the case 
%$\|\Theta^{t,\xi}\|_{2p}<\infty$ and $\|\Theta^{t,\xi'}\|_{2p}<\infty.$
$T\leq \delta(p,L)$, where the constant $\delta(p,L)$ depends on both $p$ and $L$. By applying It\^o's formula, then taking the supremum over $s$ and the expectation, we get
\begin{align*}
&\mathbb{E}\Big[\sup_{s\in[t,T]}|X_s^{t,\xi}-X_s^{t,\xi'}|^{2p}\Big]\\
\leq&\, \mathbb{E}|\xi-\xi'|^{2p}+2p\,\mathbb{E}\sup_{s\in[t,T]}\bigg( \int_t^s\big(|X_r^{t,\xi}-X_r^{t,\xi'}|^{2}\big)^{p-1}\Big[\langle X_r^{t,\xi}-X_r^{t,\xi'},b_r-b_r'\rangle dr\\
&\quad+\langle X_r^{t,\xi}-X_r^{t,\xi'},\sigma_r-\sigma_r'\rangle dW_r\Big] \bigg)+p\,\mathbb{E}\sup_{s\in[t,T]}\bigg( \int_t^s\big(|X_r^{t,\xi}-X_r^{t,\xi'}|^{2}\big)^{p-1}|\sigma_r-\sigma_r'|^2 dr \bigg)\\
&+p(p-1)\mathbb{E}\sup_{s\in[t,T]}\bigg( \int_t^s\big(|X_r^{t,\xi}-X_r^{t,\xi'}|^{2}\big)^{p-2}| (X_r^{t,\xi}-X_r^{t,\xi'})\cdot(\sigma_r-\sigma_r')|^2 dr \bigg)\\
&+\mathbb{E}\sup_{s\in[t,T]}\bigg( \int_t^s\int_E\Big(|X_{r-}^{t,\xi}-X_{r-}^{t,\xi'}+h_r-h'_r|^{2p}-|X_{r-}^{t,\xi}-X_{r-}^{t,\xi'}|^{2p}\Big)N(dr,d\theta) \bigg),
\end{align*}
where $b_\cdot:=b(\cdot,X_\cdot^{t,\xi},\mu^{t,\xi}_\cdot,Y_\cdot^{t,\xi},Z_\cdot^{t,\xi})$ and $b_\cdot':=b(\cdot,X_\cdot^{t,\xi'},\mu^{t,\xi'}_\cdot,Y_\cdot^{t,\xi'},Z_\cdot^{t,\xi'})$, the same for $\sigma,h,f$ and $g$. By using Burkholder–Davis–Gundy inequality and the positiveness of $N(dr,d\theta)$, we can find a constant $c_{p,L}$ depending on $p$ and the Lipschitz constant $L$ such that
\begin{align*}
&\mathbb{E}\sup_{s\in[t,T]}|X^{t,\xi}_s-X^{t,\xi'}_s|^{2p}\\
\leq&\, \mathbb{E}|\xi-\xi'|^{2p}+ c_{p,L}\mathbb{E}\bigg[ \int_t^T |X^{t,\xi}_r-X^{t,\xi'}_r|^{2p-1}\big(|X^{t,\xi}_r-X^{t,\xi'}_r|+|Y^{t,\xi}_r-Y^{t,\xi'}_r|+|Z^{t,\xi}_r-Z^{t,\xi'}_r|\big)dr\\
&\,\,\,+\int_t^T |X^{t,\xi}_r-X^{t,\xi'}_r|^{2p-2}\big(|X^{t,\xi}_r-X^{t,\xi'}_r|^2+|Y^{t,\xi}_r-Y^{t,\xi'}_r|^2\big)dr\\
&\,\,\,+\int_t^T\int_E|X_{r-}^{t,\xi}-X_{r-}^{t,\xi'}+h_r-h'_r|^{2p}N(dr,d\theta) \bigg]+c_{p,L}\mathbb{E}\bigg(\int_t^T|X^{t,\xi}_r-X^{t,\xi'}_r|^{4p-2}|\sigma_r-\sigma_r'|^2dr\bigg)^{\frac{1}{2}}.
\end{align*}
Noting that $W_2(\mu^{t,\xi}_r,\mu^{t,\xi'}_r)\leq \big(\mathbb{E}|X^{t,\xi}_r-X^{t,\xi'}_r|^2\big)^{\frac{1}{2}}$, a direct calculation provides 
\begin{align*}
\mathbb{E}|\Delta X_r|^{2p-1}W_2(\mu^{t,\xi}_r,\mu^{t,\xi'}_r)\leq&\,\mathbb{E}|\Delta X_r|^{2p-1}\big(\mathbb{E}|\Delta X_r|^2\big)^{\frac{1}{2}}\\\leq&\,\big(\mathbb{E}|\Delta X_r|^{2p-2}\mathbb{E}|\Delta X_r|^{2p}\mathbb{E}|\Delta X_r|^2\big)^{\frac{1}{2}}
\leq \mathbb{E}|\Delta X_r|^{2p},
\end{align*}
where $\Delta X_r:=X^{t,\xi}_r-X^{t,\xi'}_r$. 
Moreover, it follows from the Lipschitz continuity of $h$ that
\begin{align*}
\mathbb{E}\int_t^T\int_E|X_{r-}^{t,\xi}-X_{r-}^{t,\xi'}+h_r-h'_r|^{2p}N(dr,d\theta)&=\mathbb{E}\int_t^T\int_E|X_{r-}^{t,\xi}-X_{r-}^{t,\xi'}+h_r-h'_r|^{2p}\nu(d\theta)dr\\
&\leq c_{p,L}\mathbb{E}\int_t^T\big(|X_r^{t,\xi}-X_r^{t,\xi'}|^{2p}+|Y^{t,\xi}_r-Y^{t,\xi'}_r|^{2p}\big)dr.
\end{align*}
In what follows, the constant $c_{p,L}$ may change from line to line. Thanks to Young's inequality, we deduce
\begin{align}\label{eq 2p-estimation A.5.3}
\mathbb{E}\sup_{s\in[t,T]}|X^{t,\xi}_s-X^{t,\xi'}_s|^{2p}
\leq&\,\mathbb{E}|\xi-\xi'|^{2p}+c_{p,L}\mathbb{E}\bigg[\int_t^T\big(|X_r^{t,\xi}-X_r^{t,\xi'}|^{2p}+|Y^{t,\xi}_r-Y^{t,\xi'}_r|^{2p}\big)dr\bigg]\nonumber\\
&+c_{p,L}T^{p}\mathbb{E}\bigg(\int_t^T|Z^{t,\xi}_r-Z^{t,\xi'}_r|^2dr\bigg)^p.
\end{align}
On the other hand, for any $s\in[t,T]$, applying It{\^o}'s formula to $|Y^{t,\xi}_s-Y^{t,\xi'}_s|^{2p}$, we obtain 
\begin{align}\label{eq estimate in ito}
&\mathbb{E}|Y^{t,\xi}_s-Y^{t,\xi'}_s|^{2p}+\mathbb{E}\int_s^T|Y^{t,\xi}_r-Y^{t,\xi'}_r|^{2p-2}|Z^{t,\xi}_r-Z^{t,\xi'}_r|^2dr\nonumber\\
&+C_p\mathbb{E}\int_s^T\int_E |Y^{t,\xi}_r-Y^{t,\xi'}_r|^{2p-2}|H^{t,\xi}_r(\theta)-H^{t,\xi'}_r(\theta)|^2\nu(d\theta)dr\\
\leq&\, c_{p,L}\bigg( \mathbb{E}|g-g'|^{2p}+\mathbb{E}\Big[ \int_s^T|Y^{t,\xi}_r-Y^{t,\xi'}_r|^{2p-2}\langle Y^{t,\xi}_r-Y^{t,\xi'}_r,f_r-f_r'\rangle dr \Big] \bigg).\nonumber
\end{align}
It follows from a standard calculation that
\begin{align*}
\mathbb{E}\sup_{s\in[t,T]}|Y^{t,\xi}_s-Y^{t,\xi'}_s|^{2p}
\leq &\,\mathbb{E}|g-g'|^{2p}+\mathbb{E}\int_t^T |Y^{t,\xi}_r-Y^{t,\xi'}_r|^{2p-2}\langle Y^{t,\xi}_r-Y^{t,\xi'}_r,f_r-f_r' \rangle dr\\
&+c_{p,L}\mathbb{E}\bigg[\bigg(\int_t^T |Y^{t,\xi}_r-Y^{t,\xi'}_r|^{4p-2}|Z^{t,\xi}_r-Z^{t,\xi'}_r|^2dr\bigg)^{\frac{1}{2}}\bigg],
\end{align*}
where
\begin{align*}
&\mathbb{E}\bigg[\bigg(\int_t^T |Y^{t,\xi}_r-Y^{t,\xi'}_r|^{4p-2}|Z^{t,\xi}_r-Z^{t,\xi'}_r|^2dr\bigg)^{\frac{1}{2}}\bigg]\\
\leq&\, \mathbb{E}\bigg[\sup_{s\in[t,T]}|Y^{t,\xi}_r-Y^{t,\xi'}_r|^{p}\bigg(\int_t^T\,|Y^{t,\xi}_r-Y^{t,\xi'}_r|^{2p-2}|Z^{t,\xi}_r-Z^{t,\xi'}_r|^2dr\bigg)^{\frac{1}{2}}\bigg]\\
\leq&\,\Big(\mathbb{E}\sup_{s\in[t,T]}|Y^{t,\xi}_r-Y^{t,\xi'}_r|^{2p}\Big)^{\frac{1}{2}}\mathbb{E}\bigg[\int_t^T\,|Y^{t,\xi}_r-Y^{t,\xi'}_r|^{2p-2}|Z^{t,\xi}_r-Z^{t,\xi'}_r|^2dr\bigg]^{\frac{1}{2}}\\
\leq &\,\eps\mathbb{E}\sup_{s\in[t,T]}|Y^{t,\xi}_r-Y^{t,\xi'}_r|^{2p}+\eps^{-1}\mathbb{E}\bigg[\int_t^T\,|Y^{t,\xi}_r-Y^{t,\xi'}_r|^{2p-2}|Z^{t,\xi}_r-Z^{t,\xi'}_r|^2dr\bigg].
\end{align*}
Combining the previous estimate with \eqref{eq estimate in ito}, applying Young's inequality, and choosing $\eps<\frac{1}{c_{p,L}}$, we deduce 
%\begin{align*}
%&\mathbb{E}\sup_{s\in[t,T]}|Y^{t,\xi}_s-Y^{t,\xi'}_s|^{2p}+\mathbb{E}\int_s^T|Y^{t,\xi}_r-Y^{t,\xi'}_r|^{2p-2}|Z^{t,\xi}_r-Z^{t,\xi'}_r|^2dr\\
%\leq &\, c_{p,L}\bigg( \mathbb{E}|g-g'|^{2p}+\mathbb{E}\bigg[ \int_s^T(|X^{t,\xi}_r-X^{t,\xi'}_r|^{2p}+|Y^{t,\xi}_r-Y^{t,\xi'}_r|^{2p}) dr \bigg]\\
%&\qquad\quad+\mathbb{E}\bigg[\int_t^T|Y^{t,\xi}_r-Y^{t,\xi'}_r|^{2p-1}|Z^{t,\xi}_r-Z^{t,\xi'}_r|dr \bigg]\bigg).
%\end{align*}
%Then, we get
\begin{align}\label{eq 2p-estimation A.5.5}
\mathbb{E}\sup_{s\in[t,T]}|Y^{t,\xi}_s-Y^{t,\xi'}_s|^{2p}
\leq c_{p,L}\bigg(\mathbb{E}|g-g'|^{2p}+\mathbb{E}\int_t^T\big(|X^{t,\xi}_r-X^{t,\xi'}_r|^{2p}+|Y^{t,\xi}_r-Y^{t,\xi'}_r|^{2p}\big) dr \bigg).
\end{align}
Moreover, it follows the relation 
\begin{align*}
\int_t^T|Z^{t,\xi}_r-Z^{t,\xi'}_r|^2dr\leq&\, |g-g'|^2+\int_t^T\,\langle Y^{t,\xi}_r-Y^{t,\xi'}_r,f_r-f_r'\rangle dr\\
&-2\int_t^T\, \langle Y^{t,\xi}_r-Y^{t,\xi'}_r,Z^{t,\xi}_r-Z^{t,\xi'}_rdW_r\rangle
\end{align*}
that
\begin{align}\label{eq 2p-estimation A.5.6}
\mathbb{E}\bigg(\int_t^T|Z^{t,\xi}_r-Z^{t,\xi'}_r|^2dr\bigg)^p\leq c_{p,L}\mathbb{E}\bigg[ \sup_{s\in[t,T]}|X^{t,\xi}_s-X^{t,\xi'}_s|^{2p}+\sup_{s\in[t,T]}|Y^{t,\xi}_s-Y^{t,\xi'}_s|^{2p} \bigg].
\end{align}
The estimate for $H$ can be obtained in a similar way. Putting together \eqref{eq 2p-estimation A.5.3}, \eqref{eq 2p-estimation A.5.5} and \eqref{eq 2p-estimation A.5.6}, we conclude that there exists a constant $\delta(p,L) > 0$ such that, for all $T \leq \delta(p,L)$,
\begin{align}\label{eq 2p-estimation local}
\|\Theta^{t,\xi}-\Theta^{t,\xi'}\|_{2p}\leq c_{p,L}\mathbb{E}|\xi-\xi'|^{2p}.
\end{align}
Now we consider the general case where the upper bound for $T$ does not depend on $p$, meaning that the condition $T \leq \delta(p,L)$ in \eqref{eq 2p-estimation local} can be relaxed to $T \leq \delta(L)$. First, we can prove by Picard iteration that there exists a constant $\tilde{\delta}(p,L)$ such that, for any $T \leq \tilde{\delta}(p,L)$,
\begin{align}\label{eq 2p-estimation A.5.8}
\|\Theta^{t,\xi}\|_{2p}\leq \tilde{c}_{p,L}\big(\mathbb{E}|\xi|^{2p}+(\Pi_{t,T})^{p}+|g(0,\delta_0)|^{2p}\big),
\end{align}
where
\begin{align*}
\Pi_{t,T}:=&\,\bigg(\int_t^T\big(|b(r,0,\delta_0,0,0)|+|f(r,0,\delta_0,0,0)|\big)dr+\int_t^T\int_E|h(r,0,\delta_0,0,\theta)|\nu(d\theta)dr\bigg)^2\\&
+\int_t^T\|\sigma(r,0,\delta_0,0)\|^2dr.
\end{align*}
The proof of \eqref{eq 2p-estimation A.5.8} is omitted, since it follows the same argument as Theorem A.5 in \cite{delarue2002existence}. Hence, without loss of generality, we may assume $\delta(p,L)\leq \tilde{\delta}(p,L)$.\\
Second, by Proposition \ref{proposition continuity}, we have \begin{align}\label{eq 2p-estimation prop 2}
\mathbb{E}|Y^{t,\xi}_s-Y^{t,\xi'}_s|\leq C_L\big(\mathbb{E}|\xi-\xi'|^2\big)^{\frac{1}{2}}, \quad \forall s\in[t,T].
\end{align}
Consider a subdivision $t=t_0<t_1<\cdots<t_m=T$ with mesh size $\max_i |t_{i+1}-t_i| \leq \min\{\delta(p,C_L), \delta(p,L)\}$.
%Consider a subdivision of $[t,T]$ satisfying $t=t_0<t_1\cdots <t_m=T$, and
%\begin{align*}
%\forall i\in\{0,1,\ldots,m-1\},\ \ |t_{i+1}-t_i|\leq \min\{\delta(p,\sqrt{C_L}),\delta(p,L)\}.
%\end{align*}
Then, we know from \eqref{eq 2p-estimation local} and \eqref{eq 2p-estimation A.5.8} that there exist constants $c_{p,L}'$ and $\tilde{c}_{p,L}'$ such that
\begin{align*}
&\mathbb{E}\bigg[\sup_{s\in[t_{m-1},T]}|X^{t,\xi}_s-X_s^{t,\xi'}|^{2p}+\sup_{s\in[t_{m-1},T]}|Y^{t,\xi}_s-Y^{t,\xi'}_s|^{2p}+\bigg(\int_{t_{m-1}}^T|Z^{t,\xi}_s-Z^{t,\xi'}_s|^{2}ds\bigg)^p\\
&\qquad+\bigg(\int_{t_{m-1}}^T\int_E|H^{t,\xi}_s(\theta)-H^{t,\xi'}_s(\theta)|^{2}\nu(d\theta)ds\bigg)^p\bigg]
\leq\,c_{p,L}'\mathbb{E}|\xi-\xi|^{2p}
\end{align*}
and 
\begin{align}\label{eq yss0}
\mathbb{E}|Y^{s,0}_s|^{2p}\leq \tilde{c}_{p,L}'\big(|g(0,\delta_0)|^{2p}+(\Pi_{s,T})^p\big),\quad \forall s\in[t_{m-1},T].
\end{align}
Note that $(X_s^{t,\xi},Y_s^{t,\xi},Z_s^{t,\xi},H_s^{t,\xi})_{s\in[t_{m-2},t_{m-1}]}$ is the solution to the following MV-FBSDE:
\begin{equation*}
\left\{\begin{split}
X_s=&\,X^{t,\xi}_{t_{m-2}}+\int_{t_{m-2}}^sb(r,X_r,\mu_r,Y_r,Z_r)dr+\int_{t_{m-2}}^s\sigma(r,X_r,\mu_r,Y_r)dW_r\\
&+\int_{t_{m-2}}^s\int_Eh(r,X_{r-},\mu_r,Y_{r-},\theta)N(dr,d\theta),\\Y_s=&\,Y_{t_{m-1}}^{t_{m-1},X_{t_{m-1}}}+\int_s^{t_{m-1}}f(r,X_r,\mu_r,Y_r,Z_r)dr-\int_s^{t_{m-1}}Z_rdW_r-\int_s^{t_{m-1}}H_r(\theta)\tilde{N}(dr,d\theta).
\end{split}\right.
\end{equation*}
Then, by \eqref{eq 2p-estimation prop 2} and \eqref{eq yss0}, $Y^{t_{m-1},\cdot}_{t_{m-1}}$ is $C_L$-Lipschitz continuous with respect to the initial value and is bounded. Therefore, again by \eqref{eq 2p-estimation local} and \eqref{eq 2p-estimation A.5.8}, we have
\begin{align*}
&\mathbb{E}\bigg[\sup_{s\in[t_{m-2},T]}|X^{t,\xi}_s-X_s^{t,\xi'}|^{2p}+\sup_{s\in[t_{m-2},T]}|Y^{t,\xi}_s-Y^{t,\xi'}_s|^{2p}+\bigg(\int_{t_{m-2}}^T|Z^{t,\xi}_s-Z^{t,\xi'}_s|^{2}ds\bigg)^p\\
&\qquad+\bigg(\int_{t_{m-2}}^T\int_E|H^{t,\xi}_s(\theta)-H^{t,\xi'}_s(\theta)|^{2}\nu(d\theta)ds\bigg)^p\bigg]
\leq\, c_{p,L}'\mathbb{E}|\xi-\xi|^{2p}
\end{align*}
and
\begin{align*}
\mathbb{E}|Y^{s,0}_s|^{2p}\leq \tilde{c}_{p,L}'\big(|g(0,\delta_0)|^{2p}+(\Pi_{s,T})^p\big),\quad \forall s\in[t_{m-2},T].
\end{align*}
The rest of the proof is completed by induction.\\
For the second assertion, the $2p$-continuity in $x$ follows from standard arguments for classical FBSDEs (see, e.g., Theorem A.5 in \cite{delarue2002existence}), where the estimate for the jump term employs the same technique as in the first assertion. The $2p$-continuity in $\mu$ can be established in the same way as in Proposition \ref{proposition continuity eq2}.% The key observation is that the solution $\Theta^{t,x,\mu}$ depends only on the law of $\xi$. Therefore, having established continuity with respect to the random variable $\xi$ (in the $L^{2p}$ sense), {\red the continuity with respect to the measure $\mu$ is obtained by taking the infimum over all couplings $\xi, \xi'$ with laws $\mu,\mu'$. ?? we can first get the continuity with respect to $\xi$ and take the infimum to obtain the result.}
\end{proof}
\begin{remark}
From the proof of the first assertion of Theorem \ref{thm p-estimate}, we see that the small‑time condition $T \leq \delta(L)$ can be removed, provided the following assumptions hold:\\
1. (Global well‑posedness) For any $t \in [0,T]$, $\xi \in L^2(\Omega,\mathscr{F}_t,\mathbb{P};\mathbb{R}^d)$ and $x \in \mathbb{R}^d$, equations \eqref{eq MVFBSDE fully 1} and \eqref{eq MVFBSDE fully 2} admit unique solutions on $[t,T]$, denoted by $(X^{t,\xi},Y^{t,\xi},Z^{t,\xi},H^{t,\xi})$ and $(X^{t,x,\mu},Y^{t,x,\mu},Z^{t,x,\mu},H^{t,x,\mu})$, respectively.\\
2. (Uniform Lipschitz continuity of the value function) There exists a constant $C > 0$, independent of $t$, such that for all $\xi,\xi' \in L^2(\Omega,\mathscr{F}_t,\mathbb{P};\mathbb{R}^d)$, $x,x' \in \mathbb{R}^d$ and $\mu,\mu' \in \mathcal{P}_2(\mathbb{R}^d)$,
\begin{align*}
\mathbb{E} |Y^{t,\xi}_t - Y^{t,\xi'}_t|^2 \leq C \mathbb{E}|\xi-\xi'|^2, \quad
|Y^{t,x,\mu}_t - Y^{t,x',\mu'}_t| \leq C\big(|x-x'| + W_2(\mu,\mu')\big).
\end{align*}
The extension to arbitrary $T$ is then achieved by an induction argument analogous to that in Theorem \ref{thm p-estimate}, using the above uniform estimates (in particular, the first inequality).
\end{remark}
%{\red The proof is analogous to the induction argument used in Theorem \ref{thm p-estimate} by rewriting \eqref{eq 2p-estimation prop 2} as
%\begin{align*}
%\mathbb{E}\big[ |Y^{t,\xi}_t-Y^{t,\xi'}_t|^2 \big]^{\frac{1}{2}}\leq C\mathbb{E}\big[|\xi-\xi'|^2\big]^{\frac{1}{2}}.
%\end{align*}}
\begin{remark}
When the intensity measure $\nu$ is infinite, analogous well-posedness and regularity results to the above still hold. In this case, the state equation must involve the compensated Poisson random measure term
\begin{align*}
\int_E\,h(t,X_{r-},\mu_r,Y_r,\theta)\tilde{N}(dr,d\theta)
\end{align*}
to ensure the integrability. Moreover, the regularity condition on the coefficient $h$ has to be strengthened: for any $t\in[0,T],\theta\in E,x,x'\in\mathbb{R}^d,\mu,\mu'\in\mathcal{
P}_2(\mathbb{R}^d),y,y'\in\mathbb{R}$, there exists a constant $L > 0$ such that
\begin{align*}
|h(t,x,\mu,y,\theta)|\leq L(1\wedge|\theta|)
\end{align*}
and
\begin{align*}
|h(t,x,\mu,y,\theta)-h(t,x',\mu',y',\theta)|\leq L(1\wedge|\theta|)(|x-x'|+W_2(\mu,\mu')+|y-y'|).
\end{align*}
For the corresponding It\^o formula and the required estimates, we refer to \cite{hao2016mean} and \cite{li2018mean}.
\end{remark}

\subsection{First-order derivatives of $(X^{t,x,\mu},Y^{t,x,\mu}, Z^{t,x,\mu},H^{t,x,\mu})$ w.r.t. $x$}

In this subsection, we revisit the first-order derivatives of $(X^{t,x,\mu},Y^{t,x,\mu}, Z^{t,x,\mu},H^{t,x,\mu})$ with respect to  $x$. 
\begin{assumption}\label{assumption first derivative}
For any $(t,\theta)\in[0,T]\times E$, the coefficients $(b,f)(t,\cdot,\cdot,\cdot,\cdot),\ \sigma(t,\cdot,\cdot,\cdot),\ g(\cdot,\cdot)$ and $h(t,\cdot,\cdot,\cdot,\theta)$ satisfy that, for $1\leq i,j\leq d$,\\%belong to $\mathscr{C}^{1,1,1,1}_b(\mathbb{R}^d\times\mathcal{P}_2(\mathbb{R}^d)\times\mathbb{R}\times\mathbb{R}^d\to \mathbb{R}^d\times\mathbb{R}^{d\times d}\times\mathbb{R}^d\times\mathbb{R}\times\mathbb{R})$, i.e. for $1\leq i,j\leq d$,\\
(1) $\forall(x,y,z)\in\mathbb{R}^d\times\mathbb{R}\times\mathbb{R}^d$, $b_j(t,x,\cdot,y,z),\sigma_{i,j}(t,x,\cdot,y),h_j(t,x,\cdot,y,\theta),f(t,x,\cdot,y,z)$ and $g(x,\cdot)$ are in $\mathscr{C}^{1}(\mathcal{P}_2(\mathbb{R}^d))$. The first-order derivatives $\partial_\mu b_j(t,x,\mu,y,z,v)$, $\partial_\mu \sigma_{i,j}(t,x,\mu,y,v)$, $\partial_\mu f(t,x,\mu,y,z,v)$ and $\partial_\mu g(x,\mu,v)$ are jointly Lipschitz continuous in $(x,\mu,y,z,v)$ with constant $L$ and are uniformly bounded by $L$;\\
(2) $\forall\mu\in\mathcal{P}_2(\mathbb{R}^d)$, $b_j(t,\cdot,\mu,\cdot,\cdot),f(t,\cdot,\mu,\cdot,\cdot)\in C^{1,1,1}_b(\mathbb{R}^d\times\mathbb{R}\times\mathbb{R}^d)$, $\sigma_{i,j}(t,\cdot,\mu,\cdot),h(t,\cdot,\mu,\cdot,\theta)\in C^{1,1}_b(\mathbb{R}^d\times\mathbb{R})$ and $g(\cdot,\mu)\in C^1_b(\mathbb{R}^d)$. The first-order derivatives of $b,\sigma,f,g$ w.r.t. $x,y,z$ are jointly Lipschitz continuous in $(x,y,z,\mu)$ with constant $L$ and are uniformly bounded by $L$;\\ 
(3) All the first-order derivatives of $h$ w.r.t. $x,y$ are jointly Lipschitz continuous in $(x,\mu,y)$ with $L(\theta)$ and are bounded by $L(\theta)$, where the function $L(\theta)$ satisfies \eqref{eq Ltheta};\\
(4) The first-order derivative $\partial_\mu h(t,x,\mu,y,\theta,v)$ is jointly Lipschitz continuous in $(x,\mu,y,v)$ with $L(\theta)$ and is bounded by $L(\theta)$, where the function $L(\theta)$ satisfies \eqref{eq Ltheta}.
%And the derivatives $\partial_xb,\partial_x\sigma,\partial_xf,\partial_xg,\partial_yb,\partial_y\sigma,\partial_yf,\partial_zb,\partial_zf$ are bounded by $L$;\\
%(3) {\red the first-order derivatives of $b,\sigma,f,g$ w.r.t. $x,y,z,\mu$ are Lipschitz continuous with constant $L$, w.r.t. the variables $x,y,z,\mu,v$ (here $v$ is the extra variable in the Lions derivatives), and are bounded by $L$;???}\\
%(3) $|\partial_xh_j(t,x,\mu,y,\theta)-\partial_xh_j(t,x',\mu',y',\theta)|+|\partial_yh_j(t,x,\mu,y,\theta)-\partial_yh_j(t,x',\mu',y',\theta)|\leq L(\theta)(|x-x'|+W_2(\mu,\mu')+|y-y'|)$; moreover, $|\partial_\mu h_j(t,x,\mu,y,\theta,v)-\partial_\mu h_j(t,x',\mu',y',\theta,v')|\leq L(\theta)(|x-x'|+W_2(\mu,\mu')+|y-y'|+|v-v'|)$, where the function $L(\theta)$ satisfies \eqref{eq Ltheta};\\
%(4) $\forall(x,\mu,y,v)\in\mathbb{R}^d\times\mathcal{P}_2(\mathbb{R}^d)\times\mathbb{R}\times\mathbb{R}^d$, $|\partial_xh_j(t,x,\mu,y,\theta)|+|\partial_yh_j(t,x,\mu,y,\theta)|+\\|\partial_\mu h_j(t,x,\mu,y,\theta,v)|\leq L(\theta)$, where the function $L(\theta)$ satisfies \eqref{eq Ltheta}.\\
%If above conditions are satisfied, we say that $b,\sigma,f,g$ belong to $\mathscr{C}^{1,1,1,1}_b(\mathbb{R}^d\times\mathcal{P}_2(\mathbb{R}^d)\times\mathbb{R}\times\mathbb{R}^d\to \mathbb{R}^d\times\mathbb{R}^{d\times d}\times\mathbb{R}^d\times\mathbb{R}\times\mathbb{R})$.
\end{assumption}
Denote $\pi^{t,x,\mu}_r:=(r,X^{t,x,\mu}_r,\mu_r^{t,\xi},Y_r^{t,x,\mu},Z_r^{t,x,\mu})$,  $\pi^{t,x,\mu,(0)}_r:=(r,X^{t,x,\mu}_r,\mu_r^{t,\xi},Y_r^{t,x,\mu})$ for simplicity.
Let $(\{\partial_xX^{t,x,\mu,j}=(\partial_{x_i}X^{t,x,\mu,j})_{1\leq i\leq d}\}_{1\leq j\leq d},\partial_xY^{t,x,\mu}=(\partial_{x_i}Y^{t,x,\mu})_{1\leq i\leq d},\{\partial_xZ^{t,x,\mu,l}=(\partial_{x_i}Z^{t,x,\mu,l})_{1\leq i \leq d}\}_{1\leq l\leq d},\partial_xH^{t,x,\mu}=(\partial_{x_i}H^{t,x,\mu})_{1\leq i\leq d})$ be the solution to the following classical FBSDE:
\begin{align}\label{eq first x-derivative}
\partial_{x_i}X^{t,x,\mu,j}_s=&\,\delta_{i,j}+\sum_{k=1}^d\int_t^s\,\Big( \partial_{x_k}b_j(\pi^{t,x,\mu}_r)\partial_{x_i}X_r^{t,x,\mu,k}+\partial_{z_k}b_j(\pi^{t,x,\mu}_r)\partial_{x_j}Z^{t,x,\mu,k}_r \Big)dr\nonumber\\
&+\int_t^s\, \partial_y b_j(\pi_r^{t,x,\mu})\partial_{x_i}Y^{t,x,\mu}_rdr
+\sum_{k,l=1}^d\int_t^s\, \partial_{x_k}\sigma_{j,l}(\pi_r^{t,x,\mu,(0)})\partial_{x_i}X^{t,x,\mu,k}dW_r^l\nonumber\\&+\sum_{l=1}^d\int_t^s\,\partial_{x_k}\sigma_{j,l}(\pi_r^{t,x,\mu,(0)})\partial_xY^{t,x,\mu}_rdW_r^l\nonumber\\
&+\sum_{k=1}^d\int_t^s\int_E\, \partial_{x_k}h(\pi^{t,x,\mu,(0)}_{r-},\theta)\partial_{x_i}X^{t,x,\mu,k}_{r-}N(dr,d\theta)\nonumber\\
&+\int_t^s\int_E\, \partial_yh(\pi^{t,x,\mu,(0)}_{r-},\theta)\partial_{x_i}Y^{t,x,\mu}_{r-}N(dr,d\theta),\nonumber\\
\partial_{x_i}Y^{t,x,\mu}_s=&\,\sum_{k=1}^d\partial_{x_k}g(X^{t,x,\mu}_T,\mu^{t,\xi}_T)\partial_{x_i}X_T^{t,x,\mu,k}+\int_s^T\bigg( \sum_{k=1}^d\partial_{x_k}f(\pi_r^{t,x,\mu})\partial_{x_k}X^{t,x,\mu,k}_r\nonumber\\&+\partial_yf(\pi^{t,x,\mu}_r)\partial_{x_i}Y^{t,x,\mu}_r+\sum_{k=1}^d\partial_{z_k}f(\pi^{t,x,\mu}_r)\partial_{x_i}Z^{t,x,\mu,k}_r\bigg)dr\nonumber\\&-\int_t^s\int_E \partial_{x_i}H^{t,x,\mu}_r(\theta)\tilde{N}(dr,d\theta)-\sum_{l=1}^d\int_s^T\,\partial_{x_i}Z^{t,x,\mu,l}_rdW_r^l.
\end{align}
Under Assumption \ref{assumption first derivative} and for sufficiently small $T$, Theorem \ref{thm FBSDE small time duration} guarantees that FBSDE \eqref{eq first x-derivative} admits a unique solution satisfying the following property:
\begin{theorem}\label{thm first x-derivative}
Under Assumptions \ref{assumption SDE}-\ref{assumption first derivative}, for any $T \leq \delta(L)$, the $L^2$-derivative of $(X^{t,x,\mu}, Y^{t,x,\mu},$\\$Z^{t,x,\mu})$ with respect to $x$ exists and coincides with the solution $(\{\partial_xX^{t,x,\mu,j}\}_{1\leq j\leq d},\partial_xY^{t,x,\mu},$\\$\{\partial_xZ^{t,x,\mu,j}\}_{1\leq j\leq d})$ to \eqref{eq first x-derivative}, that is 
\begin{align*}
&\mathbb{E}\Big[ \sup_{s\in[t,T]}|X^{t,x+h,\mu}_s-X^{t,x,\mu}_s-\partial_{x}X_s^{t,x,\mu}h|^2 \Big]=o(|h|^2);\\
&\mathbb{E}\Big[ \sup_{s\in[t,T]}|Y^{t,x+h,\mu}_s-Y^{t,x,\mu}_s-\partial_{x}Y_s^{t,x,\mu}h|^2 \Big]=o(|h|^2);\\
&\mathbb{E}\bigg[ \int_t^T |Z^{t,x+h,\mu}_r-Z^{t,x,\mu}_r-\partial_{x}Z^{t,x,\mu}_rh|^2dr\bigg]=o(|h|^2),\\
&\mathbb{E}\bigg[ \int_t^T\int_E |H^{t,x+h,\mu}_r(\theta)-H^{t,x,\mu}_r(\theta)-\partial_{x}H^{t,x,\mu}_r(\theta)h|^2\nu(d\theta)dr\bigg]=o(|h|^2).
\end{align*}
\end{theorem}
The proof of Theorem \ref{thm first x-derivative} follows a rather standard argument; see, e.g., \cite{buckdahn2017mean, pardoux2005backward}. 
Then, using an argument analogous to that of Theorem \ref{thm p-estimate}, we establish the following result.
\begin{proposition}
Let Assumptions \ref{assumption SDE}-\ref{assumption first derivative} hold and $
T\leq\delta(L)$. Then, for any $p \geq 1$, there exists a constant $C_p > 0$ depending on the Lipschitz constants such that, for any $x,x'\in\mathbb{R}^d,\mu,\mu'\in\mathcal{P}_2(\mathbb{R}^d)$,
\begin{equation*}\begin{split}
&\|(\partial_xX^{t,x,\mu},\partial_xY^{t,x,\mu},\partial_xZ^{t,x,\mu},\partial_xH^{t,x,\mu})\|_{2p}\leq C_p,\\
&\|(\partial_xX^{t,x,\mu}-\partial_xX^{t,x',\mu'},\partial_xY^{t,x,\mu}-\partial_xY^{t,x',\mu'},\partial_xZ^{t,x,\mu}-\partial_xZ^{t,x',\mu'},\partial_xH^{t,x,\mu}-\partial_xH^{t,x',\mu'})\|_{2p}\\&\leq C_p\big(|x-x'|^{2p}+W_2(\mu,\mu')^{2p}\big).
\end{split}\end{equation*}
\end{proposition}
%Here, we adopt the Frobenius norm for matrices, denoted by $\|\cdot\|_{2p}$ in Theorem \ref{thm p-estimate}.

\subsection{First-order derivatives of $(X^{t,x,\mu},Y^{t,x,\mu}, Z^{t,x,\mu},H^{t,x,\mu})$ w.r.t. $\mu$}

In this subsection, we study the first-order derivatives of $(X^{t,x,\mu},Y^{t,x,\mu}, Z^{t,x,\mu},H^{t,x,\mu})$ w.r.t. $\mu$. 
First, we introduce the following FBSDE: for any $\eta\in L^2(\Omega,\mathscr{F}_t,\mathbb{P};\mathbb{R}^d)$, 
\begin{align}\label{eq first mu-derivative}
\mathcal{X}^{t,x,\mu}_s(\eta)=&\int_t^s\Big(\partial_xb(\pi^{t,x,\mu}_r)\mathcal{X}_r^{t,x,\mu}(\eta)+\partial_yb(\pi^{t,x,\mu}_r)\mathcal{Y}^{t,x,\mu}_r(\eta)+\partial_zb(\pi^{t,x,\mu}_r)\mathcal{Z}^{t,x,\mu}_r(\eta)\Big)dr\nonumber\\
&+\int_t^s\Big(\partial_x\sigma(\pi^{t,x,\mu,(0)}_r)\mathcal{X}_r^{t,x,\mu}(\eta)+\partial_y\sigma(\pi^{t,x,\mu,(0)}_r)\mathcal{Y}^{t,x,\mu}_r(\eta)\Big)dW_r\nonumber\\
&+\int_t^s\int_E\,\Big(\partial_x h(\pi^{t,x,\mu,(0)}_{r-},\theta)\mathcal{X}_{r-}^{t,x,\mu}(\eta)+\partial_y h(\pi^{t,x,\mu,(0)}_{r-},\theta)\mathcal{Y}^{t,x,\mu}_{r-}(\eta)\Big)N(dr,d\theta)\nonumber\\
&+\tilde{\mathbb{E}}\bigg[ \int_t^s\, \partial_\mu b(\pi^{t,x,\mu}_r,\tilde{X}^{t,\tilde{\xi}}_r)(\partial_x\tilde{X}^{t,\tilde{\xi},\mu}_r\tilde{\eta}+\tilde{\mathcal{X}}^{t,\tilde{\xi},\mu}_r(\eta))dr\nonumber\\
&\qquad+\int_t^s\, \partial_\mu \sigma(\pi^{t,x,\mu,(0)}_r,\tilde{X}^{t,\tilde{\xi}}_r)(\partial_x\tilde{X}^{t,\tilde{\xi},\mu}_r\tilde{\eta}+\tilde{\mathcal{X}}^{t,\tilde{\xi},\mu}_r(\eta))dW_r\nonumber\\
&\qquad+\int_t^s\,\int_E\, \partial_\mu h(\pi^{t,x,\mu,(0)}_{r-},\tilde{X}^{t,\tilde{\xi}}_{r-},\theta)(\partial_x\tilde{X}^{t,\tilde{\xi},\mu}_{r-}\tilde{\eta}+\tilde{\mathcal{X}}^{t,\tilde{\xi},\mu}_{r-}(\eta))N(dr,d\theta)\bigg],\nonumber\\
\mathcal{Y}^{t,x,\mu}_s(\eta)=&\,\partial_xg(X^{t,x,\mu}_T,\mu_T^{t,\xi})\mathcal{X}^{t,x,\mu}_T(\eta)+\tilde{\mathbb{E}}\Big[ \partial_\mu g(X^{t,x,\mu}_T,\mu_T^{t,\xi},\tilde{X}^{t,\tilde{\xi}}_T)(\partial_x\tilde{X}^{t,\tilde{\xi},\mu}_T\tilde{\eta}+\tilde{\mathcal{X}}^{t,\tilde{\xi},\mu}_T(\eta)) \Big]\nonumber\\
&+\int_s^T\,\Big( \partial_xf(\pi^{t,x,\mu}_r)\mathcal{X}^{t,x,\mu}_r(\eta)+\partial_yf(\pi^{t,x,\mu}_r)\mathcal{Y}^{t,x,\mu}_r(\eta)+\partial_zf(\pi^{t,x,\mu}_r)\mathcal{Z}^{t,x,\mu}_r(\eta) \Big)dr\nonumber\\
&+\tilde{\mathbb{E}}\bigg[ \int_s^T\,\partial_\mu f(\pi^{t,x,\mu}_r,\tilde{X}_r^{t,\tilde{\xi}})(\partial_x\tilde{X}^{t,\tilde{\xi},\mu}_r\tilde{\eta}+\tilde{\mathcal{X}}^{t,\tilde{\xi},\mu}_r(\eta))dr \bigg]-\int_s^T\,\mathcal{Z}^{t,x,\mu}_r(\eta)dW_r\nonumber\\
&-\int_s^T\int_E\mathcal{H}_r^{t,x,\mu}(\eta,\theta)\tilde{N}(dr,d\theta),
\end{align}
where $\tilde{X}$ denotes the independent copy of the random variable $X$, and the same applies to $\tilde{\xi}$ and $\tilde{\eta}$. 
Note that, when equations \eqref{eq MVFBSDE fully 1} and \eqref{eq MVFBSDE fully 2} have unique solutions, we have the identification $\pi^{t,\xi,\mu}:=\pi^{t,x,\mu}|_{x=\xi}=\pi^{t,\xi}$.
%Here, $(\mathcal{X}^{t,x,\mu}(\eta),\mathcal{Y}^{t,x,\mu}(\eta),\mathcal{Z}^{t,x,\mu}(\eta))$ satisfies above FBSDE \eqref{eq first mu-derivative}, 
The quadruple 
\begin{align*}
(\mathcal{X}^{t,\xi,\mu}(\eta),\mathcal{Y}^{t,\xi,\mu}(\eta),\mathcal{Z}^{t,\xi,\mu}(\eta),\mathcal{H}^{t,\xi,\mu}(\eta,\cdot)):=(\mathcal{X}^{t,x,\mu}(\eta),\mathcal{Y}^{t,x,\mu}(\eta),\mathcal{Z}^{t,x,\mu}(\eta),\mathcal{H}^{t,x,\mu}(\eta,\cdot))|_{x=\xi}
\end{align*}
satisfies the following MV-FBSDE, which is obtained from \eqref{eq first mu-derivative} by replacing $x$ with $\xi$:
\begin{align}\label{eq first mu-derivative alter}
\mathcal{X}^{t,\xi,\mu}_s(\eta)=&\int_t^s\Big(\partial_xb(\pi^{t,\xi}_r)\mathcal{X}_r^{t,\xi,\mu}(\eta)+\partial_yb(\pi^{t,\xi}_r)\mathcal{Y}^{t,\xi,\mu}_r(\eta)+\partial_zb(\pi^{t,\xi}_r)\mathcal{Z}^{t,\xi,\mu}_r(\eta)\Big)dr\nonumber\\
&+\int_t^s\Big(\partial_x\sigma(\pi^{t,\xi,(0)}_r)\mathcal{X}_r^{t,\xi}(\eta)+\partial_y\sigma(\pi^{t,\xi,(0)}_r)\mathcal{Y}^{t,\xi,\mu}_r(\eta)\Big)dW_r\nonumber\\
&+\int_t^s\int_E\,\Big(\partial_x h(\pi^{t,\xi,(0)}_{r-},\theta)\mathcal{X}_{r-}^{t,\xi,\mu}(\eta)+\partial_y h(\pi^{t,\xi,(0)}_{r-},\theta)\mathcal{Y}^{t,\xi,\mu}_{r-}(\eta)\Big)N(dr,d\theta)\nonumber\\
&+\tilde{\mathbb{E}}\bigg[ \int_t^s\, \partial_\mu b(\pi^{t,\xi}_r,\tilde{X}^{t,\tilde{\xi}}_r)(\partial_x\tilde{X}^{t,\tilde{\xi},\mu}_r\tilde{\eta}+\tilde{\mathcal{X}}^{t,\tilde{\xi},\mu}_r(\eta))dr\nonumber\\
&\qquad+\int_t^s\, \partial_\mu \sigma(\pi^{t,\xi,(0)}_r,\tilde{X}^{t,\tilde{\xi}}_r)(\partial_x\tilde{X}^{t,\tilde{\xi},\mu}_r\tilde{\eta}+\tilde{\mathcal{X}}^{t,\tilde{\xi},\mu}_r(\eta))dW_r\nonumber\\
&\qquad+\int_t^s\,\int_E\, \partial_\mu h(\pi^{t,\xi,(0)}_{r-},\tilde{X}^{t,\tilde{\xi}}_{r-},\theta)(\partial_x\tilde{X}^{t,\tilde{\xi},\mu}_{r-}\tilde{\eta}+\tilde{\mathcal{X}}^{t,\tilde{\xi},\mu}_{r-}(\eta))N(dr,d\theta)\bigg],\nonumber\\
\mathcal{Y}^{t,\xi,\mu}_s(\eta)=&\,\partial_xg(X^{t,\xi}_T,\mu_T^{t,\xi})\mathcal{X}^{t,\xi,\mu}_T(\eta)+\tilde{\mathbb{E}}\Big[ \partial_\mu g(X^{t,\xi}_T,\mu_T^{t,\xi},\tilde{X}^{t,\tilde{\xi}}_T)(\partial_x\tilde{X}^{t,\tilde{\xi},\mu}_T\tilde{\eta}+\tilde{\mathcal{X}}^{t,\tilde{\xi},\mu}_T(\eta)) \Big]\nonumber\\
&+\int_s^T\,\Big( \partial_xf(\pi^{t,\xi}_r)\mathcal{X}^{t,\xi,\mu}_r(\eta)+\partial_yf(\pi^{t,\xi}_r)\mathcal{Y}^{t,\xi,\mu}_r(\eta)+\partial_zf(\pi^{t,\xi}_r)\mathcal{Z}^{t,\xi,\mu}_r(\eta) \Big)dr\nonumber\\
&+\tilde{\mathbb{E}}\bigg[ \int_s^T\,\partial_\mu f(\pi^{t,\xi}_r,\tilde{X}_r^{t,\tilde{\xi}})(\partial_x\tilde{X}^{t,\tilde{\xi},\mu}_r\tilde{\eta}+\tilde{\mathcal{X}}^{t,\tilde{\xi},\mu}_r(\eta))dr \bigg]-\int_s^T\,\mathcal{Z}^{t,\xi,\mu}_r(\eta)dW_r\nonumber\\
&-\int_s^T\int_E \mathcal{H}^{t,\xi,\mu}_r(\eta,\theta)\tilde{N}(dr,d\theta).
\end{align}
%In equation \eqref{eq first mu-derivative}, we omit the sum of the dimensions of the vectors and matrices $i,j=1,\ldots,d$ as in equation \eqref{eq first x-derivative}. 
Under Assumption \ref{assumption first derivative}, the coefficients of MV-FBSDEs \eqref{eq first mu-derivative} and \eqref{eq first mu-derivative alter} are Lipschitz continuous and bounded. This can be verified using estimates of the type illustrated below. For instance, for any $\xi,\xi'\in L^2(\Omega,\mathscr{F}_t,\mathbb{P};\mathbb{R}^d)$, we have
\begin{align*}
\tilde{\mathbb{E}}|\partial_\mu b(\pi^{t,x,\mu}_r,\tilde{X}^{t,\xi}_r)(\tilde{\xi}-\tilde{\xi}')|\leq L\tilde{\mathbb{E}}|\tilde{\xi}-\tilde{\xi}'|\leq L(\tilde{\mathbb{E}}|\tilde{\xi}-\tilde{\xi}'|^2)^{\frac{1}{2}}.
\end{align*}
Similar estimates hold for the other coefficients.
Hence, by Theorem \ref{thm FBSDE small time duration}, MV-FBSDEs \eqref{eq first mu-derivative} and \eqref{eq first mu-derivative alter} admit unique solutions for $T \leq \delta(L)$. Then, following the argument of Theorem \ref{thm p-estimate} (or Theorem A.1 in \cite{li2018mean}), we find that for any $p \geq 1$, there exists a constant $C_p$ such that
\begin{align*}
\|(\mathcal{X}^{t,x,\mu},\mathcal{Y}^{t,x,\mu},\mathcal{Z}^{t,x,\mu},\mathcal{H}^{t,x,\mu})\|_{2p}\leq C_p.
\end{align*}
%In order to do so, we need some lemmas.
\begin{lemma}\label{lemma 2}
Suppose that Assumptions \ref{assumption SDE}-\ref{assumption first derivative} hold. Then, for any $t\in[0,T]$, $x\in\mathbb{R}^d$ and $\mu:=\mathbb{P}_{\xi}\in\mathcal{P}_2(\mathbb{R}^d)$, there exists a quadruple of processes $(\partial_\mu X^{t,x,\mu}(v),\partial_\mu Y^{t,x,\mu}(v),\partial_\mu Z^{t,x,\mu}(v),\\\partial_\mu H^{t,x,\mu}(v,\cdot))\in \mathcal{S}^2_\mathbb{F}(t,T)\times \mathcal{S}^2_\mathbb{F}(t,T)\times \mathcal{M}^2_\mathbb{F}(t,T)\times\mathcal{K}_{\nu}^2(t,T)$, depending measurably on $v\in\mathbb{R}^d$, such that 
\begin{align*}
\mathcal{X}^{t,x,\mu}_s(\eta)=&\,\overline{\mathbb{E}}[\partial_\mu X^{t,x,\mu}_s(\overline{\xi})\cdot\overline{\eta}],\ \ \mathcal{Y}^{t,x,\mu}_s(\eta)=\overline{\mathbb{E}}[\partial_\mu Y^{t,x,\mu}_s(\overline{\xi})\cdot\overline{\eta}],\ \ \mathbb{P}-a.s.,\ \forall s,\\
\mathcal{Z}^{t,x,\mu}_s(\eta)=&\,\overline{\mathbb{E}}[\partial_\mu Z^{t,x,\mu}_s(\overline{\xi})\cdot\overline{\eta}],\ \ \ dsd\mathbb{P}-a.e.,\\
\mathcal{H}^{t,x,\mu}_s(\eta,\theta)=&\,\overline{\mathbb{E}}[\partial_\mu H^{t,x,\mu}_s(\overline{\xi},\theta)\cdot\overline{\eta}],\ \ \ dsd\nu d\mathbb{P}-a.e.
\end{align*}
In particular, for any $(t,x,\mu,\xi)$, the mapping
\begin{align*}
(\mathcal{X}^{t,x,\mu}(\cdot),\mathcal{Y}^{t,x,\mu}(\cdot),\mathcal{Z}^{t,x,\mu}(\cdot),\mathcal{H}^{t,x,\mu}(\cdot)):&\,L^2(\Omega,\mathscr{F}_t,\mathbb{P};\mathbb{R}^d)\\
&\to \mathcal{S}^2_\mathbb{F}(t,T)\times \mathcal{S}^2_\mathbb{F}(t,T)\times \mathcal{M}^2_\mathbb{F}(t,T)\times\mathcal{K}_{\nu}^2(t,T)
\end{align*}
is linear and continuous.
\end{lemma}
\begin{proof}
To prove the existence, we consider the following FBSDE for each $v\in\mathbb{R}^d$:
\begin{align}\label{eq first mu-derivative-2}
&\partial_\mu X^{t,x,\mu}_s(v)=\int_t^s\Big(\partial_xb(\pi^{t,x,\mu}_r)\partial_\mu X_r^{t,x,\mu}(v)+\partial_y b(\pi^{t,x,\mu}_r)\partial_\mu Y^{t,x,\mu}_r(v)+\partial_z b(\pi^{t,x,\mu}_r)\partial_\mu Z^{t,x,\mu}_r(v)\Big)dr\nonumber\\
&+\int_t^s\Big(\partial_x\sigma(\pi^{t,x,\mu,(0)}_r)\partial_\mu X_r^{t,x,\mu}(v)+\partial_y\sigma(\pi^{t,x,\mu,(0)}_r)\partial_\mu Y^{t,x,\mu}_r(v)\Big)dW_r\nonumber\\
&+\int_t^s\int_E\Big(\partial_x h(\pi^{t,x,\mu,(0)}_{r-},\theta)\partial_\mu X_{r-}^{t,x,\mu}(v)+\partial_y h(\pi^{t,x,\mu,(0)}_{r-},\theta)\partial_\mu Y^{t,x,\mu}_{r-}(v)\Big)N(dr,d\theta)\nonumber\\
&+\tilde{\mathbb{E}}\bigg[ \int_t^s\, \Big(\partial_\mu b(\pi^{t,x,\mu}_r,\tilde{X}^{t,v,\mu}_r)\partial_x\tilde{X}^{t,v,\mu}_r+\partial_\mu b(\pi^{t,x,\mu}_r,\tilde{X}^{t,\tilde{\xi}}_r)\partial_\mu \tilde{X}^{t,\tilde{\xi},\mu}_r(v)\Big)dr\nonumber\\
&\qquad+\int_t^s\Big(\partial_\mu \sigma(\pi^{t,x,\mu,(0)}_r,\tilde{X}^{t,v,\mu}_r)\partial_x\tilde{X}^{t,v,\mu}_r+\partial_\mu \sigma(\pi^{t,x,\mu}_r,\tilde{X}^{t,\tilde{\xi}}_r)\partial_\mu \tilde{X}^{t,\tilde{\xi},\mu}_r(v)\Big)dW_r\nonumber\\
&\qquad+\int_t^s\int_E\, \Big(\partial_\mu h(\pi^{t,x,\mu,(0)}_{r-},\tilde{X}^{t,v,\mu}_{r-})\partial_x\tilde{X}^{t,v,\mu}_{r-}+\partial_\mu h(\pi^{t,x,\mu,(0)}_{r-},\tilde{X}^{t,\tilde{\xi}}_{r-})\partial_\mu \tilde{X}^{t,\tilde{\xi},\mu}_{r-}(v)\Big)N(dr,d\theta)\bigg],\nonumber\\
&\partial_\mu Y^{t,x,\mu}_s(v)=\partial_xg(X^{t,x,\mu}_T,\mu_T^{t,\xi})\partial_\mu X^{t,x,\mu}_T(v)-\int_s^T\partial_\mu Z^{t,x,\mu}_r(v)dW_r-\int_s^T\int_E \partial_\mu H^{t,x,\mu}_r(v,\theta)\tilde{N}(dr,d\theta)\nonumber\\
&+\tilde{\mathbb{E}}\Big[ \partial_\mu g(X^{t,x,\mu}_T,\mu_T^{t,\xi},\tilde{X}^{t,v,\mu}_T)\partial_x\tilde{X}^{t,v,\mu}_T+\partial_\mu g(X^{t,x,\mu}_T,\mu_T^{t,\xi},\tilde{X}^{t,\tilde{\xi}}_T)\partial_\mu \tilde{X}^{t,\tilde{\xi},\mu}_T(v) \Big]\nonumber\\
&+\int_s^T\Big( \partial_xf(\pi^{t,x,\mu}_r)\partial_\mu X^{t,x,\mu}_r(v)+\partial_yf(\pi^{t,x,\mu}_r)\partial_\mu Y^{t,x,\mu}_r(v)+\partial_zf(\pi^{t,x,\mu}_r)\partial_\mu Z^{t,x,\mu}_r(v) \Big)dr\nonumber\\
&+\tilde{\mathbb{E}}\bigg[ \int_s^T\Big(\partial_\mu f(\pi^{t,x,\mu}_r,\tilde{X}^{t,v,\mu}_r)\partial_x\tilde{X}^{t,v,\mu}_r+\partial_\mu f(\pi^{t,x,\mu}_r,\tilde{X}^{t,\tilde{\xi}}_r)\partial_\mu \tilde{X}^{t,\tilde{\xi},\mu}_r(v)\Big)dr \bigg],
\end{align}
where $\partial_\mu X^{t,\xi,\mu}(v):=\partial_\mu X^{t,x,\mu}(v)|_{x=\xi}$. 
%Here we omit the sum of the dimensions of vectors or matrices, since we know that $(\partial_\mu X_s^{t,x,\mu}(v),\partial_\mu Y_s^{t,x,\mu}(v),\partial_\mu Z_s^{t,x,\mu}(v))\in L^2(\Omega,\mathscr{F}_s,\mathbb{P};\mathbb{R}^{d\times d}\times\mathbb{R}^d,\mathbb{R}^{d\times d})$. 
It is important to note that, by a slight abuse of notation, the quadruple of processes $(\partial_\mu X^{t,x,\mu}, \partial_\mu Y^{t,x,\mu}, \partial_\mu Z^{t,x,\mu}, \partial_\mu H^{t,x,\mu})$ introduced above is merely a notation for the solution of the above FBSDE, and does not yet constitute the $\mu$-derivative of the solution to FBSDE \eqref{eq MVFBSDE fully 2}. Our aim is to prove that this quadruple in fact represents the Lions derivative of the solution to \eqref{eq MVFBSDE fully 2}.\\
Setting $x = \xi$ in \eqref{eq first mu-derivative-2} yields a new FBSDE, which is an MV-FBSDE. This new equation relates to \eqref{eq first mu-derivative-2} in the same way that \eqref{eq first mu-derivative alter} relates to \eqref{eq first mu-derivative}. We denote its solution by
\begin{align*}
(\partial_\mu X^{t,\xi,\mu}(v),&\partial_\mu Y^{t,\xi,\mu}(v),\partial_\mu Z^{t,\xi,\mu}(v),\partial_\mu H^{t,\xi,\mu}(v,\cdot))\\
&:=
(\partial_\mu X^{t,x,\mu}(v),\partial_\mu Y^{t,x,\mu}(v),\partial_\mu Z^{t,x,\mu}(v),\partial_\mu H^{t,x,\mu}(v,\cdot))|_{x=\xi}.
\end{align*}
%This is indeed a MV-FBSDE, since equation \eqref{eq first mu-derivative-2} involves the expectation of $\partial_\mu X^{t,\xi,\mu}(v)$.
According to Theorem \ref{thm FBSDE small time duration}, this MV-FBSDE and FBSDE \eqref{eq first mu-derivative-2} admit unique solutions which satisfy the following estimates:
\begin{align*}
&\|(\partial_\mu X^{t,x,\mu}(v),\partial_\mu Y^{t,x,\mu}(v),\partial_\mu Z^{t,x,\mu}(v),\partial_\mu H^{t,x,\mu}(v,\cdot))\|_{2p}\leq C_p,\\
&\|(\partial_\mu X^{t,\xi,\mu}(v),\partial_\mu Y^{t,\xi,\mu}(v),\partial_\mu Z^{t,\xi,\mu}(v),\partial_\mu H^{t,\xi,\mu}(v,\cdot))\|_{2p}\leq C_p.
\end{align*}
%Now we verify that the triple $(\partial_\mu X^{t,x,\mu}(v),\partial_\mu Y^{t,x,\mu}(v),\partial_\mu Z^{t,x,\mu}(v))$ satisfies the properties of the lemma.\\
Let $(\overline{\xi},\overline{\eta})$ and $(\tilde{\xi},\tilde{\eta})$ be defined on independent copies $(\overline{\Omega},\overline{\mathscr{F}},\overline{\mathbb{P}})$ and $(\tilde{\Omega},\tilde{\mathscr{F}},\tilde{\mathbb{P}})$ of the original probability space $(\Omega,\mathscr{F},\mathbb{P})$, respectively. We replace $v$ by $\overline{\xi}$ in \eqref{eq first mu-derivative-2}, multiply both sides by $\overline{\eta}$, and then take the expectation $\overline{\mathbb{E}}$. Note that
\begin{align*}
\overline{\mathbb{E}}\tilde{\mathbb{E}}\big[\partial_\mu b(\pi^{t,x,\mu},\tilde{X}^{t,\tilde{\xi},\mu}_r)\partial_x \tilde{X}^{t,\tilde{\xi},\mu}_r\cdot\overline{\eta}\big]&=\tilde{\mathbb{E}}\big[\partial_\mu b(\pi^{t,x,\mu},\tilde{X}^{t,\tilde{\xi},\mu}_r)\partial_x \tilde{X}^{t,\tilde{\xi},\mu}_r\cdot\tilde{\eta}\big],\\
\overline{\mathbb{E}}\tilde{\mathbb{E}}\big[\partial_\mu b(\pi^{t,x,\mu},\tilde{X}^{t,\tilde{\xi},\mu}_r)\partial_\mu \tilde{X}^{t,\tilde{\xi},\mu}_r(\overline{\xi})\cdot\overline{\eta}\big]&=\tilde{\mathbb{E}}\big[\partial_\mu b(\pi^{t,x,\mu},\tilde{X}^{t,\tilde{\xi},\mu}_r)\overline{\mathbb{E}}[\partial_\mu \tilde{X}^{t,\tilde{\xi},\mu}_r(\overline{\xi})\cdot\overline{\eta}]\big],
\end{align*}
%Then the equality between $(\partial_\mu X^{t,x,\mu}(\cdot),\partial_\mu Y^{t,x,\mu}(\cdot),\partial_\mu Z^{t,x,\mu}(\cdot))$ and $(\mathcal{X}^{t,x,\mu}(\cdot),\mathcal{Y}^{t,x,\mu}(\cdot),\mathcal{Z}^{t,x,\mu}(\cdot))$ can be obtained due to the uniqueness of FBSDEs.
and similar for the terms involving $\sigma,h,f,g$. 
%Then, comparing equation \eqref{eq first mu-derivative-2} that has undergone the above operations with equation \eqref{eq first mu-derivative}, 
Then compare with equation \eqref{eq first mu-derivative}, by the uniqueness of the solution to FBSDE, we obtain the first assertion.\\
%$$(\partial_\mu X^{t,x,\mu}(\cdot),\partial_\mu Y^{t,x,\mu}(\cdot),\partial_\mu Z^{t,x,\mu}(\cdot),\partial_\mu H^{t,x,\mu}(\cdot))=(\mathcal{X}^{t,x,\mu}(\cdot),\mathcal{Y}^{t,x,\mu}(\cdot),\mathcal{Z}^{t,x,\mu}(\cdot),\mathcal{H}^{t,x,\mu}(\cdot)).$$}
The proof of the linearity and continuity of $(\mathcal{X}^{t,x,\mu}(\cdot),\mathcal{Y}^{t,x,\mu}(\cdot),\mathcal{Z}^{t,x,\mu}(\cdot),\mathcal{H}^{t,x,\mu}(\cdot))$ follows an argument analogous to that of Lemma 6.1 in \cite{li2018mean} and is thus omitted.
\end{proof}
Next, we establish an estimate for the solution of \eqref{eq first mu-derivative-2}.
\begin{proposition}\label{estimateofeqfirstmu-derivative-2}
For any $p\geq 1$, there exists a constant $C_p>0$ depending on $p$ and $L$ such that for any $t\in[0,T],x,x',v,v'\in \mathbb{R}^d$ and $\mu,\mu'\in\mathcal{P}_2(\mathbb{R}^d)$,
\begin{align*}
\big\|\big(\partial_\mu& X^{t,x,\mu}(v)-\partial_\mu X^{t,x',\mu'}(v'),\partial_\mu Y^{t,x,\mu}(v)-\partial_\mu Y^{t,x',\mu'}(v'),\partial_\mu Z^{t,x,\mu}(v)-\partial_\mu Z^{t,x',\mu'}(v'),\\
&\qquad \partial_\mu H^{t,x,\mu}(v,\cdot)-\partial_\mu H^{t,x',\mu'}(v',\cdot)\big)\big\|_{2p}
\leq C_p\big(|x-x'|^{2p}+|v-v'|^{2p}+W_2(\mu,\mu')^{2p}\big).
\end{align*}
\end{proposition}
The proof of Proposition \ref{estimateofeqfirstmu-derivative-2} follows the same argument as that of Proposition 6.1 in \cite{li2018mean}. Specifically, we construct an MV-FBSDE whose solution is given by 
\begin{align*}
\big(\partial_\mu X^{t,x,\mu}(v)-\partial_\mu X^{t,x',\mu'}(v'),\partial_\mu Y^{t,x,\mu}(v)-\partial_\mu Y^{t,x',\mu'}(v'),\partial_\mu Z^{t,x,\mu}(v)-\partial_\mu Z^{t,x',\mu'}(v'),\\
\partial_\mu H^{t,x,\mu}(v,\cdot)-\partial_\mu H^{t,x',\mu'}(v',\cdot)\big).
\end{align*}
%Using Theorem \ref{thm p-estimate} and a standard estimate, we obtain the desired result.\\
Then, applying Theorem \ref{thm p-estimate} and a standard estimate for MV-FBSDE yields the desired result.\\

Finally, denoting by $( X^{t,x,\xi}, Y^{t,x,\xi},Z^{t,x,\xi},H^{t,x,\xi}):=( X^{t,x,\mathbb{P}_\xi}, Y^{t,x,\mathbb{P}_\xi},Z^{t,x,\mathbb{P}_\xi},H^{t,x,\mathbb{P}_\xi})$ the lifted function, we'll prove that $(\mathcal{X}^{t,x,\mu},\mathcal{Y}^{t,x,\mu},\mathcal{Z}^{t,x,\mu},\mathcal{H}^{t,x,\mu})$ is the Gateaux derivative of $( X^{t,x,\xi}, Y^{t,x,\xi},Z^{t,x,\xi},H^{t,x,\xi})$ w.r.t. $\xi$.
\begin{theorem}\label{thm first mu-derivative}
Suppose that Assumptions \ref{assumption SDE}-\ref{assumption first derivative} hold. Then, for any $t\in[0,T]$ and $x\in\mathbb{R}^d$, the mapping $\xi\mapsto( X^{t,x,\xi}, Y^{t,x,\xi},Z^{t,x,\xi},H^{t,x,\xi})$ 
%$\mathcal{S}^2_\mathbb{F}(t,T)\times \mathcal{S}^2_\mathbb{F}(t,T)\times \mathcal{M}^2_\mathbb{F}(t,T)\times\mathcal{K}_{\nu}^2(t,T)$
is Gateaux differentiable, and its Gateaux derivative at $\xi$ in the direction $\eta$ is $$(\mathcal{X}^{t,x,\mu}(\eta),\mathcal{Y}^{t,x,\mu}(\eta),\mathcal{Z}^{t,x,\mu}(\eta),\mathcal{H}^{t,x,\mu}(\eta,\cdot)),$$ that is, 
\begin{align*}
D_\xi X^{t,x,\xi}_s(\eta)=&\mathcal{X}^{t,x,\mu}_s(\eta)=\overline{\mathbb{E}}[\partial_\mu X^{t,x,\mu}_s(\overline{\xi})\cdot\overline{\eta}],\ \ \mathbb{P}-a.s.,\quad\forall s;\\
D_\xi Y^{t,x,\xi}_s(\eta)=&\mathcal{Y}^{t,x,\mu}_s(\eta)=\overline{\mathbb{E}}[\partial_\mu Y^{t,x,\mu}_s(\overline{\xi})\cdot\overline{\eta}],\ \ \mathbb{P}-a.s.,\quad\forall s;\\
D_\xi Z^{t,x,\xi}_s(\eta)=&\mathcal{Z}^{t,x,\mu}_s(\eta)=\overline{\mathbb{E}}[\partial_\mu Z^{t,x,\mu}_s(\overline{\xi})\cdot\overline{\eta}],\ \ d\mathbb{P}ds-a.e.;\\
D_\xi H^{t,x,\xi}_s(\eta,\theta)=&\mathcal{H}^{t,x,\mu}_s(\eta,\theta)=\overline{\mathbb{E}}[\partial_\mu H^{t,x,\mu}_s(\overline{\xi},\theta)\cdot\overline{\eta}],\ \ d\mathbb{P}d\nu ds-a.e.
\end{align*}
\end{theorem}
\begin{proof}
In fact, it suffices to prove that the directional derivatives of $(X^{t,x,\xi}, Y^{t,x,\xi}, Z^{t,x,\xi}, H^{t,x,\xi})$ exist in every direction $\eta \in L^2(\Omega, \mathscr{F}_t, \mathbb{P}; \mathbb{R}^d)$. That is, we need to show that the following limits exist as $\delta \to 0$:
\begin{align*}
&\mathcal{X}^{t,x,\mu}_{\cdot}(\eta)-\frac{1}{\delta}(X^{t,x,\xi+\delta\eta}_{\cdot}-X^{t,x,\xi}_{\cdot})\to 0,\ \ \text{in} \ \mathcal{S}^2_\mathbb{F};\\
&\mathcal{Y}^{t,x,\mu}_{\cdot}(\eta)-\frac{1}{\delta}(Y^{t,x,\xi+\delta\eta}_{\cdot}-Y^{t,x,\xi}_{\cdot})\to 0,\ \ \text{in} \ \mathcal{S}^2_\mathbb{F};\\
&\mathcal{Z}^{t,x,\mu}_{\cdot}(\eta)-\frac{1}{\delta}(Z^{t,x,\xi+\delta\eta}_{\cdot}-Z^{t,x,\xi}_{\cdot})\to 0,\ \ \text{in} \ \mathcal{M}^2_\mathbb{F};\\
&\mathcal{H}^{t,x,\mu}_{\cdot}(\eta,\cdot)-\frac{1}{\delta}(H^{t,x,\xi+\delta\eta}_{\cdot}-H^{t,x,\xi}_{\cdot})\to 0,\ \ \text{in} \ \mathcal{K}_{\nu}^2.
\end{align*}
By \eqref{eq first mu-derivative}, we have
\begin{align}\label{eq mu-derivative convergence 1}
&\frac{1}{\delta}(X^{t,x,\xi+\delta\eta}_s-X^{t,x,\xi}_s)-\mathcal{X}^{t,x,\mu}_s(\eta)\nonumber\\
=&\,\frac{1}{\delta}\int_t^s\Big( b(\pi^{t,x,\xi+\delta\eta}_r)-b(\pi^{t,x,\xi}_r) \Big)dr-\tilde{\mathbb{E}}\bigg[ \int_t^s\partial_\mu b(\pi^{t,x,\mu}_r,\tilde{X}^{t,\tilde{\xi}}_r)(\partial_x\tilde{X}^{t,\tilde{\xi},\mu}_r\tilde{\eta}+\tilde{\mathcal{X}}^{t,\tilde{\xi},\mu}_r(\eta))dr\bigg]\nonumber\\
&-\int_t^s\Big(\partial_xb(\pi^{t,x,\mu}_r)\mathcal{X}_r^{t,x,\mu}(\eta)+\partial_yb(\pi^{t,x,\mu}_r)\mathcal{Y}^{t,x,\mu}_r(\eta)+\partial_zb(\pi^{t,x,\mu}_r)\mathcal{Z}^{t,x,\mu}_r(\eta)\Big)dr\nonumber\\
&+\frac{1}{\delta}\int_t^s\Big( \sigma(\pi^{t,x,\xi+\delta\eta,(0)}_r)-\sigma(\pi^{t,x,\xi,(0)}_r) \Big)dW_r\nonumber\\&-\tilde{\mathbb{E}}\int_t^s\, \partial_\mu \sigma(\pi^{t,x,\mu,(0)}_r,\tilde{X}^{t,\tilde{\xi}}_r)(\partial_x\tilde{X}^{t,\tilde{\xi},\mu}_r\tilde{\eta}+\tilde{\mathcal{X}}^{t,\tilde{\xi},\mu}_r(\eta))dW_r\nonumber\\
&-\int_t^s\Big(\partial_x\sigma(\pi^{t,x,\mu,(0)}_r)\mathcal{X}_r^{t,x,\mu}(\eta)+\partial_y\sigma(\pi^{t,x,\mu,(0)}_r)\mathcal{Y}^{t,x,\mu}_r(\eta)\Big)dW_r\nonumber\\
&+\frac{1}{\delta}\int_t^s\int_E\Big( h(\pi^{t,x,\xi+\delta\eta,(0)}_{r-})-h(\pi^{t,x,\xi,(0)}_{r-}) \Big)N(dr,d\theta)\nonumber\\
&-\tilde{\mathbb{E}}\int_t^s\,\int_E\, \partial_\mu h(\pi^{t,x,\mu,(0)}_{r-},\tilde{X}^{t,\tilde{\xi}}_{r-},\theta)(\partial_x\tilde{X}^{t,\tilde{\xi},\mu}_{r-}\tilde{\eta}+\tilde{\mathcal{X}}^{t,\tilde{\xi},\mu}_{r-}(\eta))N(dr,d\theta)\nonumber\\
&-\int_t^s\int_E\Big(\partial_x h(\pi^{t,x,\mu,(0)}_{r-},\theta)\mathcal{X}_{r-}^{t,x,\mu}(\eta)+\partial_y h(\pi^{t,x,\mu,(0)}_{r-},\theta)\mathcal{Y}^{t,x,\mu}_{r-}(\eta)\Big)N(dr,d\theta).
\end{align}
Note that 
\begin{align*}
&b(\pi^{t,x,\xi+\delta\eta}_r)-b(\pi^{t,x,\xi}_r)\\=&\,\int_0^1\partial_\lambda\Big( b(r,\lambda(X^{t,x,\xi+\delta\eta}_r-X^{t,x,\xi}_r),\mu_r^{t,\xi+\delta\eta},Y^{t,x,\xi+\delta\eta}_r,Z^{t,x,\xi+\delta\eta}_r) \Big)d\lambda\\
&+\int_0^1\partial_\lambda\Big( b(r,X_r^{t,x,\xi},\lambda(\mu^{t,\xi+\delta\eta}_r-\mu_r^{t,\xi}),Y^{t,x,\xi+\delta\eta}_r,Z^{t,x,\xi+\delta\eta}_r) \Big)d\lambda\\
&+\int_0^1\partial_\lambda\Big( b(r,X_r^{t,x,\xi},\mu_r^{t,\xi},\lambda(Y^{t,x,\xi+\delta\eta}_r-Y^{t,x,\xi}_r),Z^{t,x,\xi+\delta\eta}_r) \Big)d\lambda\\
&+\int_0^1\partial_\lambda\Big( b(r,X_r^{t,x,\xi},\mu_r^{t,\xi},Y^{t,x,\xi}_r,\lambda(Z^{t,x,\xi+\delta\eta}_r-Z^{t,x,\xi}_r) \Big)d\lambda\\
=&\,\alpha_r(\delta)(X^{t,x,\xi+\delta\eta}_r-X^{t,x,\xi}_r)+\beta_r(\delta)(Y^{t,x,\xi+\delta\eta}_r-Y^{t,x,\xi}_r)\\&+\gamma_r(\delta)(Z^{t,x,\xi+\delta\eta}_r-Z^{t,x,\xi}_r)
+\tilde{\mathbb{E}}[\rho_r(\delta)(\tilde{X}^{t,\tilde{\xi}+\delta\tilde{\eta}}_r-\tilde{X}_r^{t,\tilde{\xi}})]+R_r(\delta),  
\end{align*}
where 
\begin{align*}
\alpha_r(\delta):=&\int_0^1\partial_x b(r,X^{t,x,\xi}_r+\lambda(X^{t,x,\xi+\delta\eta}_r-X^{t,x,\xi}_r),\mu_r^{t,\xi},Y_r^{t,x,\xi},Z^{t,x,\xi}_r)d\lambda,\\
\beta_r(\delta):=&\int_0^1\partial_y b(r,X^{t,x,\xi}_r,\mu_r^{t,\xi},Y_r^{t,x,\xi}+\lambda(Y^{t,x,\xi+\delta\eta}_r-Y^{t,x,\xi}_r),Z^{t,x,\xi}_r)d\lambda,\\
\gamma_r(\delta):=&\int_0^1\partial_z b(r,X^{t,x,\xi}_r,\mu_r^{t,\xi},Y_r^{t,x,\xi},Z^{t,x,\xi}_r+\lambda(Z^{t,x,\xi+\delta\eta}_r-Z^{t,x,\xi}_r))d\lambda,\\
\rho_r(\delta):=&\int_0^1\partial_\mu b(r,X^{t,x,\xi}_r,\mathbb{P}_{X^{t,\xi}_r+\lambda(X^{t,\xi+\delta\eta}_r-X^{t,\xi}_r)},Y_r^{t,x,\xi},Z^{t,x,\xi}_r,\tilde{X}^{t,\tilde{\xi}}_r+\lambda(\tilde{X}^{t,\tilde{\xi}+\delta\tilde{\eta}}_r-\tilde{X}_r^{t,\tilde{\xi}}))d\lambda,\\
R_r(\delta):=&\bigg[\int_0^1\partial_x b(r,X^{t,x,\xi}_r+\lambda(X^{t,x,\xi+\delta\eta}_r-X^{t,x,\xi}_r),\mu_r^{t,\xi+\delta\eta},Y_r^{t,x,\xi+\delta\eta},Z^{t,x,\xi+\delta\eta}_r)d\lambda-\alpha_r(\delta)\bigg]\\
&\,\cdot\big(X^{t,x,\xi+\delta\eta}_r-X^{t,x,\xi}_r\big)+\cdots.
\end{align*}
Here we omit from $R_r(\delta)$ the terms containing the $y$- and $z$-derivatives of $b$. 
The Lipschitz continuity of these derivatives, combined with Theorem \ref{thm p-estimate}, implies that
\begin{align*}
\mathbb{E}\int_t^T\,|R_r(\delta)|^2dr\leq C_L\delta^4(\mathbb{E}|\eta|^2)^2.
\end{align*}
A straightforward calculation yields 
\begin{align*}
&\frac{1}{\delta}\alpha_r(\delta)(X^{t,x,\xi+\delta\eta}_r-X^{t,x,\xi}_r)-\partial_x b(\pi_r^{t,x,\xi})\mathcal{X}^{t,x,\mu}_r(\eta)\\
=&\,\Big[\frac{1}{\delta}(X^{t,x,\xi+\delta\eta}_r-X^{t,x,\xi}_r)-\mathcal{X}^{t,x,\mu}_r\Big]\partial_x b(\pi_r^{t,x,\xi})+R_\alpha(r,\delta),
\end{align*}
where
\begin{align*}
R_\alpha(r,\delta):=\frac{1}{\delta}(\alpha_r(\delta)-\partial_x b(\pi^{t,x,\xi}_r))(X^{t,x,\xi+\delta\eta}_r-X^{t,x,\xi}_r).
\end{align*}
Then a standard computation leads to
\begin{align*}
\mathbb{E}\int_t^T\,|R_\alpha(r,\delta)|^2dr\leq \frac{1}{\delta^2}\int_t^T\,\frac{1}{2}\int_0^1\,  L\mathbb{E}|X^{t,x,\xi+\delta\eta}_r-X^{t,x,\xi}_r|^4d\lambda dr\leq C_L\delta^2(\mathbb{E}|\eta|^2)^2.
\end{align*}
We treat the remaining terms involving $\beta_r(\delta)$, $\gamma_r(\delta)$ and $\rho_r(\delta)$ in a similar manner and then substitute them into \eqref{eq mu-derivative convergence 1} to obtain
\begin{align*}
&\frac{1}{\delta}(X^{t,x,\xi+\delta\eta}_s-X^{t,x,\xi}_s)-\mathcal{X}^{t,x,\mu}_s(\eta)\\
=&\int_t^s\bigg[\partial_x b(\pi^{t,x,\xi}_r)\Big(\frac{1}{\delta}(X^{t,x,\xi+\delta\eta}_r-X^{t,x,\xi}_r)-\mathcal{X}^{t,x,\mu}_r(\eta)\Big)\\
&\qquad+\partial_y b(\pi_r^{t,x,\xi})\Big(\frac{1}{\delta}(Y^{t,x,\xi+\delta\eta}_r-Y^{t,x,\xi}_r)-\mathcal{Y}^{t,x,\mu}_r(\eta)\Big)\\
&\qquad+\partial_z b(\pi^{t,x,\xi}_r)\Big(\frac{1}{\delta}(Z^{t,x,\xi+\delta\eta}_r-Z^{t,x,\xi}_r)-\mathcal{Z}^{t,x,\mu}_r(\eta)\Big)\bigg]dr\\
&+\tilde{\mathbb{E}}\bigg[\int_t^s\partial_\mu b(\pi^{t,x,\xi}_r)\Big(\frac{1}{\delta}(\tilde{X}^{t,\tilde{\xi}+\delta\tilde{\eta}}_r-\tilde{X}_r^{t,\tilde{\xi}})-\tilde{\mathcal{X}}_r^{t,\tilde{\xi},\mu}(\eta)\Big)dr \bigg]\\%-\partial_x \tilde{X}^{t,\tilde{\xi},\mu}_r\tilde{\eta}
&+\cdots\\
&+\int_t^s\Big(\frac{1}{\delta}R_r(\delta)+R_\alpha(r,\delta)+R_\beta(r,\delta)+R_\gamma(r,\delta)+\tilde{\mathbb{E}}[R_\rho(r,\delta)]\Big)dr.
\end{align*}
The previous equation can be viewed as a forward SDE for the process $\frac{1}{\delta}(X^{t,x,\xi+\delta\eta}-X^{t,x,\xi}) - \mathcal{X}^{t,x,\mu}(\eta)$. Similarly, the triple formed by the corresponding differences of $Y$, $Z$, and $H$ constitutes a solution to a backward SDE (see Lemma 6.2 in \cite{li2018mean} for details).
%$(\frac{1}{\delta}(Y^{t,x,\xi+\delta\eta}-Y^{t,x,\xi})-\mathcal{Y}^{t,x,\mu}(\eta),\frac{1}{\delta}(Z^{t,x,\xi+\delta\eta}-Z^{t,x,\xi})-\mathcal{Z}^{t,x,\mu}(\eta),\frac{1}{\delta}(H^{t,x,\xi+\delta\eta}-H^{t,x,\xi})-\mathcal{H}^{t,x,\mu}(\eta))$ 
Combining these two parts, we deduce a coupled forward-backward SDE, whose solution is precisely the quadruple
$
(\frac{1}{\delta}(X^{t,x,\xi+\delta\eta}-X^{t,x,\xi})-\mathcal{X}^{t,x,\mu}(\eta),\frac{1}{\delta}(Y^{t,x,\xi+\delta\eta}-Y^{t,x,\xi})-\mathcal{Y}^{t,x,\mu}(\eta),\frac{1}{\delta}(Z^{t,x,\xi+\delta\eta}-Z^{t,x,\xi})-\mathcal{Z}^{t,x,\mu}(\eta),\frac{1}{\delta}(H^{t,x,\xi+\delta\eta}-H^{t,x,\xi})-\mathcal{H}^{t,x,\mu}(\eta)).
$
Then, replacing the deterministic initial condition $x$ with the random variable $\xi$,
%(i.e., considering the processes $X^{t,\xi,\cdot}$ etc., defined via the lift $\mu = \mathcal{L}(\xi)$),
%replacing $\frac{1}{\delta}(X^{t,\xi+\delta\eta}-X^{t,\xi})-\mathcal{X}^{t,\xi,\mu}(\eta)$ instead of $\frac{1}{\delta}(X^{t,x,\xi+\delta\eta}-X^{t,x,\xi})-\mathcal{X}^{t,x,\mu}(\eta)$, 
we get an MV-FBSDE whose solution is 
$(\frac{1}{\delta}(X^{t,\xi+\delta\eta}-X^{t,\xi})-\mathcal{X}^{t,\xi,\mu}(\eta),\frac{1}{\delta}(Y^{t,\xi+\delta\eta}-Y^{t,\xi})-\mathcal{Y}^{t,\xi,\mu}(\eta),
\frac{1}{\delta}(Z^{t,\xi+\delta\eta}-Z^{t,\xi})-\mathcal{Z}^{t,\xi,\mu}(\eta),\frac{1}{\delta}(H^{t,\xi+\delta\eta}-H^{t,\xi})-\mathcal{H}^{t,\xi,\mu}(\eta)).$
Finally, applying the Lipschitz conditions, the estimates for the $R-$terms and Proposition \ref{proposition continuity eq2} to this MV-FBSDE, we conclude that
\begin{align*}
\bigg\|&\Big(\frac{1}{\delta}(X^{t,\xi+\delta\eta}-X^{t,\xi})-\mathcal{X}^{t,\xi,\mu}(\eta),\frac{1}{\delta}(Y^{t,x,\xi+\delta\eta}-Y^{t,\xi})-\mathcal{Y}^{t,\xi,\mu}(\eta),\\
&\quad\frac{1}{\delta}(Z^{t,\xi+\delta\eta}-Z^{t,\xi})-\mathcal{Z}^{t,\xi,\mu}(\eta),\frac{1}{\delta}(H^{t,\xi+\delta\eta}-H^{t,\xi})-\mathcal{H}^{t,\xi,\mu}(\eta,\cdot)\Big)\bigg\|_2\leq C_L\delta^2(\mathbb{E}|\eta|^2)^2,
\end{align*}
which provides the desired result.
\end{proof}

According to Lemma 6.3 in \cite{li2018mean}, the processes $\mathcal{X}^{t,x,\mu}, \mathcal{Y}^{t,x,\mu}, \mathcal{Z}^{t,x,\mu}, \mathcal{H}^{t,x,\mu}$ are continuous in $\mu$. We have thus shown that $(\partial_\mu X^{t,x,\mu}(v), \partial_\mu Y^{t,x,\mu}(v), \partial_\mu Z^{t,x,\mu}(v), \partial_\mu H^{t,x,\mu}(v, \cdot))$ is the Lions derivative of the solution to FBSDE \eqref{eq MVFBSDE fully 2}.

\section{The second-order derivatives of the solutions}

In this section, we investigate the second-order derivatives of the solutions to equations \eqref{eq MVFBSDE fully 1} and \eqref{eq MVFBSDE fully 2}. First, we introduce the following assumption.
\begin{assumption}\label{assumption second derivative}
For any $(t,\theta)\in[0,T]\times E$, the coefficients $(b,f)(t,\cdot,\cdot,\cdot,\cdot),\ \sigma(t,\cdot,\cdot,\cdot),\ g(\cdot,\cdot)$ and $h(t,\cdot,\cdot,\cdot,\theta)$ satisfy that, for $1\leq i,j\leq d$, \\%$\in\mathscr{C}_b^{2,(1,1),2,2}(\mathbb{R}^d\times\mathcal{P}_2(\mathbb{R}^d)\times\mathbb{R}\times\mathbb{R}^d\to \mathbb{R}^d\times\mathbb{R}^{d\times d}\times\mathbb{R}^d\times\mathbb{R}\times\mathbb{R})$, i.e. for $1\leq i,j\leq d$, it holds\\
(1) $\forall(x,y,z)\in\mathbb{R}^d\times\mathbb{R}\times\mathbb{R}^d$, $b_j(t,x,\cdot,y,z),\sigma_{i,j}(t,x,\cdot,y),h_j(t,x,\cdot,y,\theta),f(t,x,\cdot,y,z),g(x,\cdot)$ are in $\mathscr{C}^{1,1}(\mathcal{P}_2(\mathbb{R}^d))$. Moreover, the derivatives $\partial_\mu b(t,\cdot,\cdot,\cdot,\cdot,\cdot)$ and $\partial_v\partial_\mu b(t,\cdot,\cdot,\cdot,\cdot,\cdot)$ are jointly Lipschitz continuous in $(x,\mu,y,z,v)$ with constant $L$ and are bounded by $L$. Similar for the functions $\sigma,f$ and $g$.\\
%\begin{align*}
%|b(t,x,\mu,y,z)-b(t,x',\mu',y',z')|\leq& L(|x-x'|+W_2(\mu,\mu')+|y-y'|+|z-z'|),\\
%|\partial_\mu b(t,x,\mu,y,z,v)-\partial_\mu b(t,x',\mu',y',z',v')|\leq& L(|x-x'|+W_2(\mu,\mu')+|y-y'|+|z-z'|\\
%\qquad&+|v-v'|),\\
%|\partial_v\partial_\mu b(t,x,\mu,y,z,v)-\partial_v\partial_\mu b(t,x',\mu',y',z',v')|\leq& L(|x-x'|+W_2(\mu,\mu')+|y-y'|+|z-z'|\\
%\qquad&+|v-v'|),\\
%\end{align*}
%and
%\begin{align*}
%&|b(t,x,\mu,y,z)|\leq L,\ \
%|\partial_\mu b(t,x,\mu,y,z,v)|\leq L,\\
%&|\partial_v\partial_\mu b(t,x,\mu,y,z,v)|\leq L,
%\end{align*}
%for any $x,x'z,z',v,v'\in\mathbb{R}^d\ ,y,y'\in\mathbb{R},\ \mu,\mu'\in\mathcal{P}_2(\mathbb{R}^d)$. And similar for $\sigma,f,g$.\\
(2) $\forall\mu\in\mathcal{P}_2(\mathbb{R}^d)$, $b_j(t,\cdot,\mu,\cdot,\cdot),f(t,\cdot,\mu,\cdot,\cdot)\in C^{2,2,2}_b(\mathbb{R}^d\times\mathbb{R}\times\mathbb{R}^d)$, $\sigma_{i,j}(t,\cdot,\mu,\cdot),h(t,\cdot,\mu,\cdot,\theta)\in C^{2,2}_b(\mathbb{R}^d\times\mathbb{R})$ and $g(\cdot,\mu)\in C^2_b(\mathbb{R}^d)$. All the first- and second-order derivatives of $b,\sigma,f,g$ w.r.t. $x,y,z$ are jointly Lipschitz continuous in  $(x,y,z,\mu)$ with $L$ and are bounded by $L$.\\
%(3) all the first and second order derivatives of $b,\sigma,f,g$ and themselves are Lipschitz continuous w.r.t. the variables $x,y,z,\mu,v$ (here $v$ is the extra variable in the Lions derivatives) with constant $L$ and bounded by $L$;\\
(3) All the first- and second-order derivatives of $h(t,\cdot,\cdot,\cdot,\theta)$ w.r.t. $x,y$ are jointly Lipschitz continuous in $(x,y,\mu)$ with $L(\theta)$ and are bounded by $L(\theta)$, where $L(\theta)$ satisfies \eqref{eq Ltheta}.\\
(4) The derivatives $\partial_\mu h(t,x,\mu,y,\theta,v)$ and $\partial_v\partial_\mu h(t,x,\mu,y,\theta,v)$ are jointly Lipschitz continuous in  $(x,y,\mu,v)$ with $L(\theta)$ and are bounded by $L(\theta)$, where  $L(\theta)$ satisfies \eqref{eq Ltheta}.
\end{assumption}
\begin{theorem}\label{thm second derivative}
Let Assumptions \ref{assumption SDE}-\ref{assumption y lip}, \ref{assumption second derivative} hold and $T\leq \delta(L)$. Then, for any $t\in[0,T]$, $x,v\in\mathbb{R}^d$ and $\mu:=\mathbb{P}_\xi\in\mathcal{P}_2(\mathbb{R}^d)$,\\
(1) the mappings
\begin{align*}
x\mapsto(\partial_xX^{t,x,\mu},\partial_xY^{t,x,\mu},\partial_xZ^{t,x,\mu},\partial_xH^{t,x,\mu})
\end{align*}
and 
\begin{align*}
v\mapsto(\partial_\mu X^{t,x,\mu}(v),\partial_\mu Y^{t,x,\mu}(v),\partial_\mu Z^{t,x,\mu}(v),\partial_\mu H^{t,x,\mu}(v,\cdot))
\end{align*}
are $L^2-$differentiable;\\
(2) for any $p\geq 1$, there exists a constant $C_p>0$ depending on $p$, $L$ and $M$, s.t. for $(\mathscr{X}^{t,x,\mu}(v),\mathscr{Y}^{t,x,\mu}(v),\mathscr{Z}^{t,x,\mu}(v),\mathscr{H}^{t,x,\mu}(v,\cdot))\in\{(\partial^2_xX^{t,x,\mu},\partial^2_xY^{t,x,\mu},\partial^2_xZ^{t,x,\mu},\partial^2_xH^{t,x,\mu}),\\(\partial_v\partial_\mu X^{t,x,\mu}(v),\partial_v\partial_\mu Y^{t,x,\mu}(v),\partial_v\partial_\mu Z^{t,x,\mu}(v),\partial_v\partial_\mu H^{t,x,\mu}(v,\cdot))\}$, the following estimates hold
\begin{align*}
\|(\mathscr{X}^{t,x,\mu}(v),\mathscr{Y}^{t,x,\mu}(v),\mathscr{Z}^{t,x,\mu}(v),\mathscr{H}^{t,x,\mu}(v,\cdot))\|_{2p}\leq C_p
\end{align*}
and 
\begin{align*}
\|&(\mathscr{X}^{t,x,\mu}(v)-\mathscr{X}^{t,x',\mu'}(v'),\mathscr{Y}^{t,x,\mu}(v)-\mathscr{Y}^{t,x',\mu'}(v'),\mathscr{Z}^{t,x,\mu}(v)-\mathscr{Z}^{t,x',\mu'}(v'),\\
&\qquad\mathscr{H}^{t,x,\mu}(v,\cdot)-\mathscr{H}^{t,x',\mu'}(v',\cdot))\|_{2p}
\leq C_p(|x-x'|^{2p}+|v-v'|^{2p}+W_2(\mu,\mu')^{2p}).
\end{align*}
\end{theorem}
\begin{proof}
The proof for $(\partial^2_xX^{t,x,\mu},\partial^2_xY^{t,x,\mu},\partial^2_xZ^{t,x,\mu},\partial^2_xH^{t,x,\mu})$ is exactly the same as in Theorem \ref{thm first x-derivative}. For $(\partial_v\partial_\mu X^{t,x,\mu}(v),\partial_v\partial_\mu Y^{t,x,\mu}(v),\partial_v\partial_\mu Z^{t,x,\mu}(v),\partial_v\partial_\mu H^{t,x,\mu}(v,\cdot))$, we construct the following FBSDE:
\begin{align}\label{eq second y mu derivative}
\partial_v\partial_\mu& X^{t,x,\mu}_s(v)
=\,\int_t^s\Big(\partial_x b(\pi^{t,x,\mu}_r)\partial_v\partial_\mu X^{t,x,\mu}_r+\partial_y b(\pi^{t,x,\mu}_r)\partial_v\partial_\mu Y^{t,x,\mu}_r+\partial_z b(\pi^{t,x,\mu}_r)\partial_v\partial_\mu Z^{t,x,\mu}_r\Big)dr\nonumber\\
&+\int_t^s\tilde{\mathbb{E}}\bigg[\partial_\mu b(\pi_r^{t,x,\mu},\tilde{X}^{t,v,\mu}_r)\partial_x^2 \tilde{X}^{t,v,\mu}_r+\partial_v\partial_\mu b(\pi^{t,x,\mu}_r,\tilde{X}^{t,v,\mu}_r)(\partial_x\tilde{X}^{t,v,\mu}_r)^2 \bigg]dr\nonumber\\
&+\int_t^s\tilde{\mathbb{E}}\bigg[ \partial_\mu b(\pi^{t,x,\mu}_r,\tilde{X}^{t,\tilde{\xi}}_r)\partial_v\partial_\mu \tilde{X}^{t,\tilde{\xi},\mu}_r(v) \bigg]dr+\ldots,\nonumber\\
\partial_v\partial_\mu& Y^{t,x,\mu}_s(v)=\,\partial_xg(X^{t,x,\mu}_T,\mu^{t,\xi}_T)\partial_v\partial_\mu X^{t,x,\mu}_T(y)-\int_s^T\partial_v\partial_\mu Z^{t,x,\mu}_r(v)dW_r\nonumber\\
&+\int_s^T\Big(\partial_x f(\pi^{t,x,\mu}_r)\partial_v\partial_\mu X^{t,x,\mu}_r+\partial_y f(\pi^{t,x,\mu}_r)\partial_v\partial_\mu Y^{t,x,\mu}_r+\partial_z f(\pi^{t,x,\mu}_r)\partial_v\partial_\mu Z^{t,x,\mu}_r\Big)dr\nonumber\\
&+\int_s^T\tilde{\mathbb{E}}\bigg[\partial_\mu f(\pi_r^{t,x,\mu},\tilde{X}^{t,v,\mu}_r)\partial_x^2 \tilde{X}^{t,v,\mu}_r+\partial_v\partial_\mu f(\pi^{t,x,\mu}_r,\tilde{X}^{t,v,\mu}_r)(\partial_x\tilde{X}^{t,v,\mu}_r)^2 \bigg]dr\nonumber\\
&+\int_s^T\tilde{\mathbb{E}}\bigg[ \partial_\mu f(\pi^{t,x,\mu}_r,\tilde{X}^{t,\tilde{\xi}}_r)\partial_v\partial_\mu \tilde{X}^{t,\tilde{\xi},\mu}_r(v) \bigg]dr-\int_s^T\partial_v\partial_\mu H^{t,x,\mu}_r(\theta)\tilde{N}(dr,d\theta),
\end{align}
where we omit the terms corresponding to $\sigma$ and $h$, as they are similar to those for $b$. Following the same procedure as in Lemma \ref{lemma 2}, we set $x = \xi$ in \eqref{eq second y mu derivative} and obtain an MV-FBSDE solved by $(\partial_v\partial_\mu X^{t,\xi,\mu}(v),\partial_v\partial_\mu Y^{t,\xi,\mu}(v),\partial_v\partial_\mu Z^{t,\xi,\mu}(v),\partial_v\partial_\mu H^{t,\xi,\mu}(v,\cdot))$, which, by construction, is equal to $(\partial_v\partial_\mu X^{t,x,\mu}(v),\partial_v\partial_\mu Y^{t,x,\mu}(v),\partial_v\partial_\mu Z^{t,x,\mu}(v),\partial_v\partial_\mu H^{t,x,\mu}(v,\cdot))|_{x=\xi}.$ Then, by Theorem \ref{thm FBSDE small time duration} and Theorem \ref{thm p-estimate}, the MV-FBSDE and FBSDE \eqref{eq second y mu derivative} admit unique solutions, and these solutions satisfy the $2p-$boundedness and continuity.\\
Finally, to verify that $(\partial_v\partial_\mu X^{t,x,\mu}(v),\partial_v\partial_\mu Y^{t,x,\mu}(v),\partial_v\partial_\mu Z^{t,x,\mu}(v),\partial_v\partial_\mu H^{t,x,\mu}(v,\cdot))$ is indeed the derivative of $(\partial_\mu X^{t,x,\mu}(v),\partial_\mu Y^{t,x,\mu}(v),\partial_\mu Z^{t,x,\mu}(v),\partial_\mu H^{t,x,\mu}(v,\cdot))$, we follow the argument in Theorem \ref{thm first mu-derivative} and compute the $L^2$-convergence of the following expressions:
\begin{align*}
&\frac{1}{\delta}(\partial_\mu X^{t,x,\mu}(v+\delta)-\partial_\mu X^{t,x,\mu}(v)-\partial_v\partial_\mu X^{t,x,\mu}(v)),\\&\frac{1}{\delta}(\partial_\mu Y^{t,x,\mu}(v+\delta)-\partial_\mu Y^{t,x,\mu}(v)-\partial_v\partial_\mu Y^{t,x,\mu}(v)),\\&\frac{1}{\delta}(\partial_\mu Z^{t,x,\mu}(v+\delta)-\partial_\mu Z^{t,x,\mu}(v)-\partial_v\partial_\mu Z^{t,x,\mu}(v)),\\&\frac{1}{\delta}(\partial_\mu H^{t,x,\mu}(v+\delta,\cdot)-\partial_\mu H^{t,x,\mu}(v,\cdot)-\partial_v\partial_\mu H^{t,x,\mu}(v,\cdot)).
\end{align*}
This completes the proof.
\end{proof}

\section{Related integral-master equations}
In this section, we establish the relation between the MV-FBSDEs \eqref{eq MVFBSDE fully 1}, \eqref{eq MVFBSDE fully 2} and the master equation \eqref{eq master introduction game} that arises from mean-field games with Poisson jumps.\\
Set the decoupled field of FBSDE \eqref{eq MVFBSDE fully 2} as $V(t,x,\mu):=Y^{t,x,\mu}_t$. Then we have
\begin{proposition}\label{prop time continuity}
Under Assumption \ref{assumption SDE}-\ref{assumption y lip} and \ref{assumption second derivative}, and for $T\leq \delta(L)$, the function $V$ has the following properties:\\
(1) For any $t\in[0,T]$, $V(t,\cdot,\cdot)\in\mathscr{C}^{2,(1,1)}(\mathbb{R}^d\times\mathcal{P}_2(\mathbb{R}^d))$, i.e. $V(t,x,\cdot)\in\mathscr{C}^{1,1}(\mathcal{P}_2(\mathbb{R}^d))$ for any $x\in\mathbb{R}^d$ and $V(t,\cdot,\mu)\in C^2(\mathbb{R}^d)$ for any $\mu\in\mathcal{P}_2(\mathbb{R}^d)$.\\
(2) For any $x\in\mathbb{R}^d$ and $\mu\in\mathcal{P}_2(\mathbb{R}^d)$, $V(\cdot,x,\mu)$ is uniformly Lipschitz continuous in $t$ (hence belongs to $C([0,T])$).\\
(3) Let $\varphi:=\partial_x V,\partial_x^2 V,\partial_\mu V, \partial_v\partial_\mu V$. Then there exist constants $C>0$ and $q\ge 1$ (independent of $x,\mu,v,\varphi$) such that for all $t,t'\in[0,T]$, $x,v\in\mathbb{R}^d$ and $\mu\in\mathcal{P}_2(\mathbb{R}^d)$,
\begin{align*}
|\varphi(t,x,\mu,v)-\varphi(t',x,\mu,v)|\leq C|t-t'|^{\frac{1}{q}}.
\end{align*}
\end{proposition}
\begin{proof}
Claim (1) follows directly from Theorems \ref{thm first x-derivative}, \ref{thm first mu-derivative} and \ref{thm second derivative}.  Claim (2) and (3) follow from Lemma A.2 in \cite{li2018mean}; note that the dependence of the forward SDE coefficients on $y$ and $z$ does not affect the validity of the estimates.
\end{proof}
\begin{theorem}
Under Assumptions \ref{assumption SDE}-\ref{assumption y lip}, \ref{assumption second derivative} and for $T\leq \delta(L)$, $V\in\mathscr{C}^{1,2,(1,1)}([0,T]\times\mathbb{R}^d\times\mathcal{P}_2(\mathbb{R}^d))$ is the unique classical solution to master equation \eqref{eq master introduction game}.
\end{theorem}
\begin{proof}
By Proposition \ref{prop time continuity}, $V(t,\cdot,\cdot)\in\mathscr{C}^{2,(1,1)}(\mathbb{R}^d\times\mathcal{P}_2(\mathbb{R}^d))$ for any $t$. Moreover, Theorem 9.1 in \cite{li2018mean} implies that $V$ is differentiable in $t$. Therefore, we can apply the It\^o–Lions formula (Corollary 3.5 in \cite{guo2023ito}) to $V(s,X^{t,x,\mu}_s,\mu_s^{t,\xi})$, where $\mu^{t,\xi}_s:=\mathbb{P}_{X^{t,\xi}_s}$, and obtain
\begin{align*}
&dV(s,X^{t,x,\mu}_s,\mu_s^{t,\xi})\\
=&\,\bigg\{\partial_t V(s,X^{t,x,\mu}_s,\mu_s^{t,\xi})+\partial_x V(s,X^{t,x,\mu}_s,\mu_s^{t,\xi})b(\pi_s^{t,x,\mu})+\frac{1}{2}\partial^2_x V(s,X^{t,x,\mu}_s,\mu_s^{t,\xi}):\sigma\sigma^{\intercal}(\pi^{t,x,\mu,(0)}_s)\\
&\,\,+\tilde{\mathbb{E}}\Big[ \partial_\mu V(s,X^{t,x,\mu}_s,\mu_s^{t,\xi},\tilde{X}^{t,\xi}_s)b(\tilde{\pi}^{t,x,\mu}_s)+\frac{1}{2}\partial_v\partial_\mu V(s,X^{t,x,\mu}_s,\mu_s^{t,\xi}):\sigma\sigma^{\intercal}(\tilde{\pi}^{t,x,\mu,(0)}_s)\\
&\,\,+\int_E\Big(\frac{\delta V}{\delta\mu}(s,X^{t,x,\mu}_s,\mu_s^{t,\xi},\tilde{X}^{t,x,\mu}_s+h(\tilde{\pi}^{t,x,\mu,(0)}_s,\theta))-\frac{\delta V}{\delta\mu}(s,X^{t,x,\mu}_s,\mu_s^{t,\xi},\tilde{X}^{t,x,\mu}_s)\Big)\nu(d\theta)\Big]\bigg\}ds\\&+\partial_x V(s,X^{t,x,\mu}_s,\mu_s^{t,\xi})\sigma(\pi^{t,x,\mu,(0)}_s)dW_s\\
&+\int_E\Big( V(s,X^{t,x,\mu}_s+h(\pi^{t,x,\mu,(0)}_s,\theta),\mu_s^{t,\xi})-V(s,X^{t,x,\mu}_s,\mu_s^{t,\xi}) \Big)N(ds,d\theta).
\end{align*}
From the definition of $V$ and the BSDE satisfied by $Y^{t,x,\mu}_s = Y^{s,X^{t,x,\mu}_s,\mu^{t,\xi}_s}_s$, it follows immediately that
\begin{align*}
Z^{t,x,\mu}_s&=\partial_xV(s,X^{t,x,\mu}_s,\mu_s^{t,\xi})\sigma(\pi^{t,x,\mu,(0)}_s),\\
H^{t,x,\mu}_s(\theta)&=V(s,X^{t,x,\mu}_s+h(\pi^{t,x,\mu,(0)}_s,\theta),\mu_s^{t,\xi})-V(s,X^{t,x,\mu}_s,\mu_s^{t,\xi}),
\end{align*}
hence, we have
\begin{align*}
&\int_E\Big( V(s,X^{t,x,\mu}_s+h(\pi^{t,x,\mu,(0)}_s,\theta),\mu_s^{t,\xi})-V(s,X^{t,x,\mu}_s,\mu_s^{t,\xi}) \Big)N(ds,d\theta)\\
=&\int_E H^{t,x,\mu}_s(\theta)\tilde{N}(ds,d\theta)+\int_E\Big( V(s,X^{t,x,\mu}_s+h(\pi^{t,x,\mu,(0)}_s,\theta),\mu_s^{t,\xi})-V(s,X^{t,x,\mu}_s,\mu_s^{t,\xi}) \Big)\nu(d\theta)ds.
\end{align*}
Setting $s=t$ yields
\begin{align*}
0=&\,\partial_tV(t,x,\mu)+\partial_x V(t,x,\mu)b(t,x,\mu,V(t,x,\mu),\partial_xV(t,x,\mu)\sigma(t,x,\mu,V(t,x,\mu)))\\
&+f(t,x,\mu,V(t,x,\mu),\partial_xV(t,x,\mu)\sigma(t,x,\mu,V(t,x,\mu)))\\
&+\frac{1}{2}\partial_x^2V(t,x,\mu):\sigma\sigma^{\intercal}(t,x,\mu,V(t,x,\mu))\\&+\int_E\, (V(t,x+h(t,x,\mu,V(t,x,\mu)),\mu),\theta)-V(t,x,\mu))\nu(d\theta)\\
&+\int_{\mathbb{R}^d}\, \bigg[\partial_\mu V(t,x,\mu,y)b(t,y,\mu,V(t,y,\mu),\partial_xV(t,y,\mu)\sigma(t,y,\mu,V(t,y,\mu)))\\
&\qquad\qquad+\frac{1}{2}\partial_y\partial_\mu V(t,x,\mu,y):\sigma\sigma^{\intercal}(t,y,\mu,V(t,y,\mu))\\
&\qquad\qquad+\int_E\Big(\frac{\delta V}{\delta\mu}(t,x,\mu,y+h(t,y,\mu,V(t,y,\mu),\theta))-\frac{\delta V}{\delta\mu}(t,x,\mu,y)\Big)\nu(d\theta)\bigg]\mu(dy).
\end{align*}
To prove the uniqueness, we assume that there exists another solution $V'$ to the master equation \eqref{eq master introduction game}. Then, by the It\^o-Lions formula, it is easy to verify that $V'(s, X^{t,x,\mu}_s,\mu^{t,\xi}_s)$ satisfies
\begin{align*}
&dV'(s,X^{t,x,\mu}_s,\mu_s^{t,\xi})\\
=&\,\bigg\{ \partial_t V'(s,X^{t,x,\mu}_s,\mu_s^{t,\xi})+\partial_x V'(s,X^{t,x,\mu}_s,\mu_s^{t,\xi})b(\pi_s^{t,x,\mu})+\frac{1}{2}\partial^2_x V'(s,X^{t,x,\mu}_s,\mu_s^{t,\xi}):\sigma\sigma^{\intercal}(\pi^{t,x,\mu,(0)}_s)\\
&\,\,+\tilde{\mathbb{E}}\Big[ \partial_\mu V'(s,X^{t,x,\mu}_s,\mu_s^{t,\xi},\tilde{X}^{t,\xi}_s)b(\tilde{\pi}^{t,x,\mu}_s)+\frac{1}{2}\partial_v\partial_\mu V'(s,X^{t,x,\mu}_s,\mu_s^{t,\xi}):\sigma\sigma^{\intercal}(\tilde{\pi}^{t,x,\mu,(0)}_s)\\
&\,\,+\int_E\Big(\frac{\delta V'}{\delta\mu}(s,X^{t,x,\mu}_s+h(\tilde{\pi}^{t,x,\mu,(0)}_s,\theta),\mu_s^{t,\xi},\tilde{X}^{t,x,\mu}_s)-\frac{\delta V'}{\delta\mu}(s,X^{t,x,\mu}_s,\mu_s^{t,\xi},\tilde{X}^{t,x,\mu}_s)\Big)\nu(d\theta)\Big]\bigg\}ds\\&\,+\partial_x V'(s,X^{t,x,\mu}_s,\mu_s^{t,\xi})\sigma(\pi^{t,x,\mu,(0)}_s)dW_s\\
&+\int_E\Big( V'(s,X^{t,x,\mu}_s+h(\pi^{t,x,\mu,(0)}_s,\theta),\mu_s^{t,\xi})-V'(s,X^{t,x,\mu}_s,\mu_s^{t,\xi}) \Big)
N(ds,d\theta).
\end{align*}
Since $V'$ satisfies equation \eqref{eq master introduction game} as also, we obtain
\begin{align*}
&dV'(s,X^{t,x,\mu}_s,\mu_s^{t,\xi})\\
=&\,f(s,X^{t,x,\mu}_s,\mu^{t,\xi}_s,V'(s,X^{t,x,\mu}_s,\mu_s^{t,\xi}),\partial_xV'(s,X^{t,x,\mu}_s,\mu_s^{t,\xi})\sigma(s,X^{t,x,\mu}_s,\mu^{t,\xi}_s,V'(s,X^{t,x,\mu}_s,\mu_s^{t,\xi})))ds\\
&+\partial_xV'(s,X^{t,x,\mu}_s,\mu_s^{t,\xi})\sigma(s,X^{t,x,\mu}_s,\mu^{t,\xi}_s,V'(s,X^{t,x,\mu}_s,\mu_s^{t,\xi}))dW_s\\
&+\int_E\Big( V'(s,X^{t,x,\mu}_s+h(s,X^{t,x,\mu}_s,V'(s,X^{t,x,\mu},\mu^{t,\xi}_s),\theta),\mu_s^{t,\xi})-V'(s,X^{t,x,\mu}_s,\mu_s^{t,\xi}) \Big)\tilde{N}(ds,d\theta).
\end{align*}
Therefore, the following quadruple
\begin{align*}
(X^{t,x,\mu}_\cdot,V'(\cdot,X^{t,x,\mu}_\cdot,\mu_\cdot^{t,\xi}),\partial_xV'(\cdot,X^{t,x,\mu}_\cdot,\mu_\cdot^{t,\xi})\sigma(\cdot,X^{t,x,\mu},\mu^{t,\xi}_\cdot,V'(\cdot,X^{t,x,\mu}_\cdot,\mu_\cdot^{t,\xi})),\\
V'(s,X^{t,x,\mu}_s+h(s,X^{t,x,\mu}_s,V'(s,X^{t,x,\mu},\mu^{t,\xi}_s),\theta),\mu_s^{t,\xi})-V'(s,X^{t,x,\mu}_s,\mu_s^{t,\xi}))
\end{align*}
is also a solution to FBSDE \eqref{eq MVFBSDE fully 2}. By the uniqueness result of FBSDEs, we have 
\begin{align*}
V'(s,X^{t,x,\mu}_s,\mu_s^{t,\xi})=Y^{t,x,\mu}_s=V(s,X^{t,x,\mu}_s,\mu_s^{t,\xi}),\quad \forall s.
\end{align*}
In particular, setting $s=t$ yields
$V'(t,x,\mu)=V(t,x,\mu),$
which completes the proof.
\end{proof}

\bibliographystyle{abbrv}
\bibliography{Feynman-Kac_formula_for_master_equation}

@article{chassagneux2014probabilistic,
  title={A probabilistic approach to classical solutions of the master equation for large population equilibria},
  author={Chassagneux, Jean-Fran{\c{c}}ois and Crisan, Dan and Delarue, Fran{\c{c}}ois},
  journal={Memoirs of the American Mathematical Society},
  volume={280},
  number={1379},
  year={2022},
  pages={123 pages},
  publisher={American Mathematical Society}
}

@book{carmona2018probabilistic,
  title={Probabilistic theory of mean field games with applications I-II},
  author={Carmona, Ren{\'e} and Delarue, Fran{\c{c}}ois},
  year={2018},
  publisher={Springer}
}

@book{cardaliaguet2010notes,
  title={Notes on mean field games},
  author={Cardaliaguet, Pierre},
  year={2010},
  publisher={Technical report}
}

@article{pardoux1999forward,
  title={Forward-backward stochastic differential equations and quasilinear parabolic {PDE}s},
  author={Pardoux, Etienne and Tang, Shanjian},
  journal={Probability {T}heory and {R}elated {F}ields},
  volume={114},
  pages={123--150},
  year={1999},
  publisher={Springer}
}

@article{gangbo2019differentiability,
  title={On differentiability in the {W}asserstein space and well-posedness for {H}amilton--{J}acobi equations},
  author={Gangbo, Wilfrid and Tudorascu, Adrian},
  journal={Journal de Math{\'e}matiques Pures et Appliqu{\'e}es},
  volume={125},
  pages={119--174},
  year={2019},
  publisher={Elsevier}
}

@article{rudiger2004stochastic,
  title={Stochastic integration with respect to compensated Poisson random measures on separable Banach spaces},
  author={R{\"u}diger, Barbara},
  journal={Stochastics and Stochastic Reports},
  volume={76},
  number={3},
  pages={213--242},
  year={2004},
  publisher={Taylor \& Francis}
}

@book{applebaum2009levy,
  title={L{\'e}vy processes and stochastic calculus},
  author={Applebaum, David},
  year={2009},
  publisher={Cambridge University Press}
}

@article{huang2006large,
  title={Large population stochastic dynamic games: closed-loop {M}c{K}ean-{V}lasov systems and the {N}ash certainty equivalence principle},
  journal={Communications in {I}nformation and {S}ystems},
  author={Huang, Minyi and Malham{\'e}, Roland P and Caines, Peter E},
  pages={221--252},
  volume={6},
  number={3},
  year={2006}
}

@article{lasry2007mean,
  title={Mean field games},
  author={Lasry, Jean-Michel and Lions, Pierre-Louis},
  journal={Japanese Journal of Mathematics},
  volume={2},
  number={1},
  pages={229--260},
  year={2007},
  publisher={Springer}
}

@article{lions2007cours,
  title={\textit{Cours au {C}oll{\`e}ge de {F}rance}},
  author={Lions, Pierre-Louis},
  journal={\rm{Available at www. college-de-france. fr}},
  year={2007}
}

@article{carmona2014master,
  title={The master equation for large population equilibriums},
  author={Carmona, Ren{\'e} and Delarue, Francois},
  journal={Stochastic Analysis and Applications 2014},
  pages={77--128},
  year={2014},
  publisher={Springer}
}

@book{cardaliaguet2019master,
  title={The master equation and the convergence problem in mean field games},
  author={Cardaliaguet, Pierre and Delarue, Fran{\c{c}}ois and Lasry, Jean-Michel and Lions, Pierre-Louis},
  year={2019},
  publisher={Princeton University Press}
}

@article{mou2024mean,
  title={Mean field game master equations with anti-monotonicity conditions},
  author={Mou, Chenchen and Zhang, Jianfeng},
  journal={Journal of the European Mathematical Society},
  volume={27},
  number={11},
  pages={4469–4499},
  year={2024}
}

@article{meszaros2024mean,
  title={Mean field games systems under displacement monotonicity},
  author={M{\'e}sz{\'a}ros, Alp{\'a}r R and Mou, Chenchen},
  journal={SIAM Journal on Mathematical Analysis},
  volume={56},
  number={1},
  pages={529--553},
  year={2024},
  publisher={SIAM}
}

@article{guo2023ito,
  title={It{\^o}’s formula for flows of measures on semimartingales},
  author={Guo, Xin and Pham, Huy{\^e}n and Wei, Xiaoli},
  journal={Stochastic {P}rocesses and their {A}pplications},
  volume={159},
  pages={350--390},
  year={2023},
  publisher={Elsevier}
}

@article{chassagneux2019numerical,
  title={Numerical method for {FBSDE}s of {M}c{K}ean--{V}lasov type},
  author={Chassagneux, Jean-Fran{\c{c}}ois and Crisan, Dan and Delarue, Fran{\c{c}}ois},
  journal={The Annals of Applied Probability},
  volume={29},
  number={3},
  pages={1640--1684},
  year={2019},
  publisher={JSTOR}
}

@article{jisheng2025wentzell,
  title={It\^{o}-{W}entzell-{L}ions formulae for flows of full and conditional measures on semimartingales},
  author={Liu, Jisheng and Zhang, Jing},
  journal={arXiv preprint arXiv:2505.13155},
  year={2025}
}

@article{li2018mean,
  title={Mean-field forward and backward {SDE}s with jumps and associated nonlocal quasi-linear integral-{PDE}s},
  author={Li, Juan},
  journal={Stochastic Processes and their Applications},
  volume={128},
  number={9},
  pages={3118--3180},
  year={2018},
  publisher={Elsevier}
}

@article{zhang2006wellposedness,
  title={The wellposedness of {FBSDE}s},
  author={Zhang, Jianfeng},
  journal={Discrete and Continuous Dynamical Systems Series B},
  volume={6},
  number={4},
  pages={927},
  year={2006},
  publisher={AIMS PRESS}
}

@article{delarue2002existence,
  title={On the existence and uniqueness of solutions to {FBSDE}s in a non-degenerate case},
  author={Delarue, Fran{\c{c}}ois},
  journal={Stochastic {P}rocesses and their {A}pplications},
  volume={99},
  number={2},
  pages={209--286},
  year={2002},
  publisher={Elsevier}
}

@inproceedings{pardoux2005backward,
  title={Backward stochastic differential equations and quasilinear parabolic partial differential equations},
  author={Pardoux, Etienne and Peng, Shige},
  booktitle={Stochastic Partial Differential Equations and Their Applications: Proceedings of IFIP WG 7/1 International Conference University of North Carolina at Charlotte, NC June 6--8, 1991},
  pages={200--217},
  year={2005},
  organization={Springer}
}

@article{buckdahn2017mean,
  title={Mean-field stochastic differential equations and associated {PDE}s},
  author={Buckdahn, Rainer and Li, Juan and Peng, Shige and Rainer, Catherine},
  journal={The Annals of Probability},
  volume={45},
  number={2},
  pages={824--878},
  year={2017}
}

@article{hao2016mean,
  title={Mean-field {SDE}s with jumps and nonlocal integral-{PDE}s},
  author={Hao, Tao and Li, Juan},
  journal={Nonlinear Differential Equations and Applications},
  volume={23},
  number={2},
  pages={17 pages},
  year={2016},
  publisher={Springer}
}

@article{barles1997backward,
  title={Backward stochastic differential equations and integral-partial differential equations},
  author={Barles, Guy and Buckdahn, Rainer and Pardoux, Etienne},
  journal={Stochastics: An International Journal of Probability and Stochastic Processes},
  volume={60},
  number={1-2},
  pages={57--83},
  year={1997},
  publisher={Taylor \& Francis}
}

@article{chen2015semi,
  title={SEMI-LINEAR BACKWARD STOCHASTIC INTEGRAL PARTIAL DIFFERENTIAL EQUATIONS DRIVEN BY A {B}ROWNIAN MOTION AND A {P}OISSON POINT PROCESS.},
  author={Chen, Shaokuan and Tang, Shanjian},
  journal={Mathematical Control \& Related Fields},
  volume={5},
  number={3},
  pages={401--434},
  year={2015}
}

\end{document}